\documentclass[a4paper,10pt]{amsart}
\usepackage[utf8]{inputenc}
\usepackage{amsxtra}
\usepackage{amsopn}
\usepackage{amsmath,amsthm,amssymb}
\usepackage{amscd}
\usepackage{amsfonts}
\usepackage{latexsym}
\usepackage{verbatim}
\usepackage{MnSymbol}
\usepackage{xcolor}
\usepackage{tikz-cd}
\usepackage{hyperref}
\usepackage[shortlabels]{enumitem}

\newcommand{\la}{\langle}
\newcommand{\ra}{\rangle}
\newcommand{\lv}{\lVert}
\newcommand{\rv}{\rVert}
\theoremstyle{plain}
\newtheorem{theorem}{Theorem}[section]
\newtheorem*{theorem*}{Theorem}

\newtheorem{lemma}[theorem]{Lemma}
\newtheorem{proposition}[theorem]{Proposition}
\newtheorem{corollary}[theorem]{Corollary}
\newtheorem{remark}[theorem]{Remark}

\newtheorem*{mt*}{Main Theorem}
\newcommand\C{{\mathbb C}}

\renewcommand\phi{{\varphi}}

\newcommand\N{{\mathbb N}}

\newcommand\R{{\mathbb R}}
\newcommand\Z{{\mathbb Z}}

\renewcommand\H{{\mathcal H}}

\newcommand{\del}{\partial}
\newcommand{\delbar}{{\overline{\del}}}

\newcommand{\cinf}{\mathcal{C}^\infty}
\DeclareMathOperator{\supp}{supp}
\let\inf\undefined
\DeclareMathOperator*{\inf}{inf\vphantom{p}}

\DeclareMathOperator{\vol}{Vol}

\DeclareMathOperator{\lcm}{lcm}

\DeclareMathOperator{\im}{im\,}

\DeclareMathOperator{\D}{\mathcal{D}}

\DeclareMathOperator{\rank}{rank}
\DeclareMathOperator{\Hom}{Hom}
\DeclareMathOperator{\lV}{\lVert}
\DeclareMathOperator{\rV}{\rVert}
\DeclareMathOperator{\gr}{Gr}

\let\phi\varphi
\let\c\overline

\title{The $L^2$ Aeppli-Bott-Chern Hilbert complex}

\author{Tom Holt}
\address{Department of Mathematical Sciences\\
University of Bath\\
Claverton Down\\
Bath\\
BA2 7AY\\
United Kingdom}
\email{th2276@bath.ac.uk}
\author{Riccardo Piovani}
\address{Dipartimento di Scienze Matematiche, Fisiche e Informatiche\\
Unit\`{a} di Matematica e Informatica\\
Universit\`{a} degli Studi di Parma\\
Parco Area delle Scienze 53/A \\
43124 Parma, Italy}
\email{riccardo.piovani@unipr.it}

\keywords{$L^2$ Hodge theory, complex manifolds, complete K\"ahler metrics, elliptic operators, spectral gap, Galois coverings}
\thanks{\newline 
The second author is partially supported by GNSAGA of INdAM}
\subjclass[2020]{53C55, 32Q15}

\begin{document}

\begin{abstract}
We analyse the $L^2$ Hilbert complexes naturally associated to a non-compact complex manifold, namely the ones which originate from the Dolbeault and the Aeppli-Bott-Chern complexes. In particular we define the $L^2$ Aeppli-Bott-Chern Hilbert complex and examine its main properties on general Hermitian manifolds, on complete K\"ahler manifolds and on Galois coverings of compact complex manifolds. The main results are achieved through the study of self-adjoint extensions of various differential operators whose kernels, on compact Hermitian manifolds, are isomorphic to either Aeppli or Bott-Chern cohomology.
\end{abstract}
\maketitle
\tableofcontents
\section{Introduction}
A \emph{Hilbert complex} consists of a complex of mutually orthogonal Hilbert spaces $\H_i$ along with linear operators $D_i:\H_i\to\H_{i+1}$ which are \emph{densely defined} ($D_i$ is defined on a domain $\D(D_i)$ which is a dense subspace of $\H_i$) and \emph{closed} (the graph of $D_i$ is closed in $\H_i\times\H_{i+1}$) of the form
\[
0\longrightarrow\H_0\overset{D_0}{\longrightarrow}\H_1\longrightarrow\dots\longrightarrow\H_{n-1}\overset{D_{n-1}}{\longrightarrow}\H_n\longrightarrow 0
\]
such that $\im D_i\subseteq \ker D_{i+1}$. The notion of \emph{Hilbert complexes} was systematically analysed by Br\"uning and Lesch in \cite{BL}. In this paper we are interested in studying the Hilbert complexes naturally associated to a non-compact complex manifold without boundary. 

Recall that on a complex manifold $M$ the exterior derivative on $(p,q)$-forms splits as
\[
d=\del+\delbar: A^{p,q}\longrightarrow A^{p+1,q}\oplus A^{p,q+1},
\]
where
\[
\del:A^{p,q}\longrightarrow A^{p+1,q},\ \ \ \delbar:A^{p,q}\longrightarrow  A^{p,q+1},
\]
thus the relation $d^2=0$ implies $\del^2=\delbar^2=\del\delbar+\delbar\del=0$. Therefore the following cohomologies
\[
H^{\bullet,\bullet}_\delbar=\frac{\ker\delbar}{\im\delbar},\ \ \ H^{\bullet,\bullet}_{BC}=\frac{\ker\del\cap\ker\delbar}{\im\del\delbar},\ \ \ H^{\bullet,\bullet}_{A}=\frac{\ker\del\delbar}{\im\del+\im\delbar},
\]
called respectively \emph{Dolbeault}, \emph{Bott-Chern} and \emph{Aeppli cohomology spaces} are well-defined.
The elliptic complexes to which these cohomologies are associated are respectively the \emph{Dolbeault complex}
\[
\dots\longrightarrow A^{p,q-1}\overset{\delbar}{\longrightarrow} A^{p,q}\overset{\delbar}{\longrightarrow} A^{p,q+1}{\longrightarrow}\dots
\]
and the \emph{Aeppli-Bott-Chern complex}, or \emph{ABC complex} for short,
\[
\dots\longrightarrow A^{p-1,q-2}\oplus A^{p-2,q-1}\overset{\delbar\oplus\del}{\longrightarrow} A^{p-1,q-1}\overset{\del\delbar}{\longrightarrow} A^{p,q}\overset{\del+\delbar}{\longrightarrow} A^{p+1,q}\oplus A^{p,q+1}{\longrightarrow}\dots
\]
where $\delbar\oplus\del$ operates on $A^{p-1,q-2}\oplus A^{p-2,q-1}$ as $\delbar$ on $A^{p-1,q-2}$ plus $\del$ on $A^{p-2,q-1}$. In particular the \emph{Aeppli complex} is given by the differentials $\delbar\oplus\del$ and $\del\delbar$, while the \emph{Bott-Chern complex} is given by the differentials $\del\delbar$ and $\del+\delbar$. The ABC complex first appeared in \cite{Bi}.

If we fix a Hermitian metric $g$ on the complex manifold $M$ with complex dimension $n$, then $\del$ and $\delbar$ have $L^2$ formal adjoint operators
\[
\del^*:A^{p,q}\longrightarrow A^{p-1,q},\ \ \ \delbar^*:A^{p,q}\longrightarrow  A^{p,q-1},
\]
defined by $\del^*:=-*\delbar*$ and $\delbar^*:=-*\del*$, where $*:A^{p,q}\to A^{n-q,n-p}$ is the $\C$-linear Hodge operator with respect to $g$.
The Laplacian which is naturally associated to the Dolbeault complex is the \emph{Dolbeault Laplacian}
\[
\Delta_\delbar=\delbar\delbar^*+\delbar^*\delbar,
\]
which is elliptic and formally self-adjoint, while to the Bott-Chern complex (for the Aeppli complex the situation is similar) there are multiple naturally associated Laplacians. The first is obtained as the standard Laplacian which is associated to any complex, namely
\begin{equation*}\label{eq non ell lapl bc}
\Delta_{BC}=\del\delbar\delbar^*\del^*+\del^*\del+\delbar^*\delbar,
\end{equation*}
which is formally self-adjoint but not elliptic \cite{S}. Kodaira and Spencer, in the proof of the stability of the K\"ahler condition under small deformations of the complex structure \cite{KS}, introduced another differential operator
\begin{equation*}\label{eq ell lapl bc ks}
\tilde\Delta_{BC}=
\del\delbar\delbar^*\del^*+
\delbar^*\del^*\del\delbar+\del^*\delbar\delbar^*\del+\delbar^*\del\del^*\delbar
+\del^*\del+\delbar^*\delbar,
\end{equation*}
usually referred as the \emph{Bott-Chern Laplacian},
which is elliptic and formally self-adjoint and, when the manifold is compact, has the same kernel as $\Delta_{BC}$. We will be also interested in a third Laplacian which is associated to the Bott-Chern complex, that is
\begin{equation*}\label{eq ell lapl bc v}
\square_{BC}=\del\delbar\delbar^*\del^*+(\del^*\del+\delbar^*\delbar)^2,
\end{equation*}
which is the elliptic and formally self-adjoint operator naturally associated to the Bott-Chern complex viewed as an elliptic complex \cite{V}. When the manifold is compact $\square_{BC}$ has the same kernel of $\Delta_{BC}$ and $\tilde\Delta_{BC}$.

In order to obtain a Hilbert complex starting from a geometric complex like the ones just mentioned, one first restricts every differential $P$ of the complex to an operator $P_0$ defined on the space of smooth compactly supported forms, and then extends the restricted differential $P_0$ to a densely defined and closed operator on the space of $L^2$ forms with respect to a chosen metric: the two most important closed extensions are called \emph{strong} and \emph{weak}, denoted respectively by $P_s$ and $P_w$, which correspond respectively to the minimal and the maximal closed extensions defined in a distributional sense of the operator $P_0$. The $L^2$ Hodge theory related to the Hilbert complex originated from the Dolbeault complex, which from now on will be called \emph{$L^2$ Dolbeault Hilbert complex}, has been studied even before \cite{BL}, in parallel with the development of the $L^2$ Hodge theory associated to the de Rham complex: we refer to \cite{AV,H65,FK,Che,De82,DF,G} as a partial list of milestones on this subject.

On the other hand, the study of a Hilbert complex arising from the ABC complex is absent in the literature.
The aim of this paper is to define an \emph{$L^2$ Aeppli-Bott-Chern Hilbert complex}, \emph{$L^2$ ABC Hilbert complex} for short, and to establish its fundamental properties.
In this setting the main difficulties arise from the fact that in the ABC complex the differential $\del\delbar$ is of second order, while in the classical Hodge-de Rham or Dolbeault complexes all the differentials are of first order. For example, it is well known that on a complete Riemannian or Hermitian manifold, $P_s=P_w$ and $P^*_s=P^*_w$ (where $P^*$ is the formal adjoint of $P$) for many interesting first order differential operators like $P=d,\del,\delbar$ \cite[Proposition 5]{AV}. Furthermore, the integer powers of $P+P^*$ (including the Dolbeault Laplacian $\Delta_{\delbar}$ and the Hodge Laplacian) are essentially self-adjoint \cite[Section 3.B]{Che}, \emph{i.e.}, they have a unique self-adjoint extension.

The main results of this paper are contained in Sections \ref{sec hilb complex bc a}, \ref{sec complete kahler}, \ref{sec coverings}. In Section \ref{sec hilb complex bc a} we define the $L^2$ ABC Hilbert complex
\begin{equation*}
\begin{tikzcd}
L^2\Lambda^{p-1,q-2}\oplus L^2\Lambda^{p-2,q-1}\arrow[d,"(\delbar \oplus \del)_a"]\\
L^2\Lambda^{p-1,q-1}\arrow[d,"\del\delbar_b"]\\
L^2\Lambda^{p,q}\arrow[d,"(\del+\delbar)_c"]\\
L^2\Lambda^{p+1,q}\oplus L^2\Lambda^{p,q+1}
\end{tikzcd}
\end{equation*}
where $a,b,c\in\{s,w\}$ with $a\le b\le c$ denote either strong or weak extensions with the order relation $s\le w$. The associated $L^2$ Aeppli and Bott-Chern cohomologies are
\[
L^2H^{p-1,q-1}_{A,ab}:=\frac{L^2\Lambda^{p-1,q-1}\cap\ker\del\delbar_b}{L^2\Lambda^{p-1,q-1}\cap\im(\delbar \oplus \del)_a},\ \ \ 
L^2H^{p,q}_{BC,bc}:=\frac{L^2\Lambda^{p,q}\cap\ker (\del+\delbar)_c}{L^2\Lambda^{p,q}\cap\im\del\delbar_b},
\]
while the reduced $L^2$ Aeppli and Bott-Chern cohomologies are defined as
\[
L^2\bar{H}^{p-1,q-1}_{A,ab}:=\frac{L^2\Lambda^{p-1,q-1}\cap\ker\del\delbar_b}{L^2\Lambda^{p-1,q-1}\cap\c{\im(\delbar \oplus \del)_a}},\ \ \ 
L^2\bar{H}^{p,q}_{BC,bc}:=\frac{L^2\Lambda^{p,q}\cap\ker (\del+\delbar)_c}{L^2\Lambda^{p,q}\cap\c{\im\del\delbar_b}}.
\]

We now list the main results of the paper (we state them just for the Bott-Chern case since the Aeppli case is analogous):
\begin{theorem}[Section \ref{sec hilb complex bc a}]\label{thm sec 7}
Given a Hermitian manifold $(M,g)$, we set the space of $L^2$ Bott-Chern harmonic forms to be $L^2\H^{p,q}_{BC,bc}:=\ker (\del+\delbar)_c\cap\ker \delbar^*\del^*_{b'}$, where $b'=s$ if $b=w$ and $b'=w$ if $b=s$.
\begin{itemize}
\item The space of $L^2$ $(p,q)$-forms has an orthogonal decomposition
$$
L^2\Lambda^{p,q}=L^2\H^{p,q}_{BC,bc}\oplus\c{\im\del\delbar_b}\oplus\c{\im(\del^*\oplus\delbar^*)_{c'}}.
$$
\item The reduced $L^2$ Bott-Chern cohomology is isomorphic to the space of $L^2$ Bott-Chern harmonic forms, \textit{i.e.}, $L^2\bar{H}^{p,q}_{BC,bc}\simeq L^2\H^{p,q}_{BC,bc}$.
\item The $L^2$ Bott-Chern cohomology $L^2H^{p,q}_{BC,bc}$ can be computed via a natural smooth subcomplex.
\item If $\square_{BC}$ is essentially self adjoint, then $\del\delbar_s=\del\delbar_w$.
\item There exists a diagram of maps between spaces of 
reduced $L^2$ Bott-Chern, Dolbeault, de Rham, $\del$ and Aeppli cohomologies.
\end{itemize}
\end{theorem}
The first three points are an application of the general theory of \cite{BL} to the specific case of the $L^2$ ABC complex, the fourth point is a generalisation of \cite[Lemma 3.8]{BL}, while the last point generalises the well known diagram of maps between usual cohomology spaces \cite{DGMS}.

In Section \ref{sec complete kahler}, we focus on complete K\"ahler manifolds. Recall that a Hermitian metric is called K\"ahler when its fundamental form is closed, and that on a compact complex manifold endowed with a K\"ahler metric the kernels of the Bott-Chern, Aeppli and Dolbeault Laplacians coincide. By Hodge theory, this implies that Bott-Chern, Aeppli and Dolbeault cohomology spaces are isomorphic and the $\del\delbar$-Lemma holds. We prove the following
\begin{theorem}[Section \ref{sec complete kahler}]\label{thm sec 8}
Let $(M,g)$ be a complete K\"ahler manifold.
\begin{itemize}
\item The space of $L^2$ Bott-Chern harmonic forms coincides with the space of $L^2$ Aeppli, Dolbeault, Hodge harmonic forms.
\item All the maps in the diagram of Theorem \ref{thm sec 7} are isomorphisms.
\item There exists an $L^2$ reduced $\del\delbar$-Lemma: if $\alpha$ is an $L^2$ $k$-form lying in $\ker\del_w\cap\ker\delbar_w$, then
\begin{align*}
\alpha\in\c{\im\del\delbar_s}&\iff\alpha\in\c{\im d_s} \iff\alpha\in\c{\im\del_s}\\
&\iff\alpha\in\c{\im\delbar_s}\iff \alpha\in\c{\im\delbar_s} +\c{\im\del_s}.
\end{align*}
\item If the unique self-adjoint extension of $\Delta_\delbar$ has a spectral gap, then the reduced and unreduced $L^2$ Bott-Chern cohomologies coincide and  $\del\delbar_s=\del\delbar_w$.
\end{itemize}
\end{theorem}
The first three points are generalisations of similar well-known results in the compact K\"ahler setting, while the last point is inspired by \cite[Theorem 1.4.A]{G}.
The results are all obtained through the study of suitable self adjoint extensions of the Bott-Chern Laplacian $\tilde{\Delta}_{BC}$.

Finally, in Section \ref{sec coverings}, we study the special setting of a Galois covering of a compact complex manifold $\pi:\widetilde M\to M\simeq \widetilde M/\Gamma$. By \cite[Proposition 3.1]{Ati}, any lift to the covering $\widetilde M$ of an elliptic and formally self-adjoint operator on the compact manifold $M$ is essentially self-adjoint, therefore by the fourth point of Theorem \ref{thm sec 7} we obtain $\del\delbar_s=\del\delbar_w$ on $\widetilde M$. This ultimately allows us to define \emph{$L^2$ Bott-Chern numbers} $h^{p,q}_{BC,\Gamma}(M)$ and \emph{$L^2$ Aeppli numbers} $h^{p,q}_{A,\Gamma}(M)$ of the Galois covering (independently from the choice of the metric on the compact manifold), generalising the usual Bott-Chern and Aeppli numbers of a compact complex manifold. \textit{E.g.}, $h^{p,q}_{BC,\Gamma}(M)$ is defined as the Von Neumann dimension of 
$L^2\H^{p,q}_{BC,sw}$ on $\widetilde M$.
Denoting by $h^{p,q}_{\delbar,\Gamma}(M)$ and $h^{p,q}_{\del,\Gamma}(M)$ the $L^2$ Hodge numbers, we prove the following inequality.
\begin{theorem}
Given a Galois covering of a compact complex manifold $\pi:\widetilde M\to M\simeq \widetilde M/\Gamma$, it holds that
\[
h^{p,q}_{\del,\Gamma}(M)+h^{p,q}_{\delbar,\Gamma}(M)\le h^{p,q}_{A,\Gamma}(M)+h^{p,q}_{BC,\Gamma}(M).
\]
\end{theorem}
This is a generalisation of the same inequality for the usual Hodge, Bott-Chern and Aeppli numbers on a compact complex manifold, established by Angella and Tomassini in \cite[Theorem A]{AT}.

We remark that all three Laplacians $\Delta_{BC}$, $\tilde\Delta_{BC}$ and $\square_{BC}$ enter in play in different arguments in our treatment. The operator $\Delta_{BC}$ is needed to apply the general theory of \cite{BL} to our case; on the other hand $\tilde\Delta_{BC}$ is useful on K\"ahler manifolds to prove that $L^2$ Bott-Chern harmonic forms coincide with $L^2$ Dolbeault harmonic forms in Theorem \ref{thm sec 8}; finally $\square_{BC}$ is fundamental when dealing with Galois coverings of compact complex manifolds: it is elliptic and formally self-adjoint (allowing us to apply \cite[Proposition 3.1]{Ati}) and it is naturally defined from the ABC complex (allowing us to prove the fourth point of Theorem \ref{thm sec 7}).

We refer the reader to the following recent papers on $L^2$ Hodge and cohomology theory on non-compact (almost) Hermitian manifolds \cite{B2,B3,HT,R,PT3,HuTa,MPa,MPb,P}.
Furthermore, we mention \cite{PT1,PT2}, where Tomassini and the second author prove a characterisation of $L^2$ smooth forms which are in the kernel of the Bott-Chern Laplacian on special families of Stein manifolds and of complete Hermitian manifolds; \cite{TWZ}, where Tan, Wang and Zhou study $L^2$ Bott-Chern and Aeppli symplectic cohomologies and harmonic forms on complete non-compact almost K\"ahler manifolds with bounded geometry and introduce the $L^2$ $dd^{\Lambda}$ Lemma; \cite{BDET}, where Bei, Diverio, Eyssidieux and Trapani generalise the notion of K\"ahler hyperbolicity introduced by Gromov in \cite{G}, proving a spectral gap result for the Dolbeault Laplacian under suitable modification.

The paper is structured in the following way.
In Section \ref{sec diff op} we define differential operators, Lebesgue spaces and elliptic complexes on manifolds.
In Section \ref{sec complex manifold} we describe the elliptic complexes naturally defined on a complex manifold and their associated Laplacians.
In Section \ref{sec hilb} we give a brief survey of the theory of unbounded linear operators on Hilbert spaces, including a statement of the Spectral Theorem and a study of the spectral gap condition for a positive self-adjoint operator.
In Section \ref{sec extensions} we define the minimal and maximal closed extensions of differential operators, collecting the main properties which will be used in the following sections.
In Section \ref{sec hilb complex dolbeault} we recall the fundamental definitions concerning the $L^2$ Dolbeault Hilbert complex.
Finally, in Section \ref{sec hilb complex bc a}, \ref{sec complete kahler}, \ref{sec coverings}  we prove the main results of the paper. We conclude in Section \ref{sec questions} with some open questions and final remarks.

In this manuscript, our aim was to include a good amount of the necessary preliminaries, with the intention of making it readable for mathematicians who are not familiar with $L^2$ Hodge theory and spectral theory.

\medskip\medskip
\noindent{\em Acknowledgments.} 
We would like to thank Lorenzo Sillari, Jonas Stelzig, Adriano Tomassini and Weiyi Zhang for interesting comments and conversations. Furthermore, we are particularly grateful to Francesco Bei for his time spent answering our many questions and providing useful suggestions which have greatly improved this paper.

\section{Differential operators and elliptic complexes}\label{sec diff op}
In this section we follow mainly \cite[Chapter VI, Section 1]{De} and \cite[Section I]{V}.
Let $M$ be a differentiable manifold of dimension $m$, and let $E_1,E_2$ be $\C$-vector bundles over $M$, with $\rank E_i=r_i$, $i=1,2$. Denote by $\Gamma(M,E_i)$ the spaces of smooth sections $M\to E_i$.
A $\C$-linear \emph{differential operator} of order $l$ from $E_1$ to $E_2$ is a $\C$-linear operator $P:\Gamma(M,E_1)\to \Gamma(M,E_2)$ locally given by
\begin{equation*}
Pu(x)=\sum_{\lv\alpha\rv\le l}a_\alpha(x)D^\alpha u(x)\ \ \ \forall u\in\Gamma(M,E_1),
\end{equation*}
where $\alpha=(\alpha_1,\dots,\alpha_m)$ is a multi-index with norm $\lv\alpha\rv=\alpha_1+\dots+\alpha_m$,
the functions 
\begin{equation*}
a_\alpha(x)=(a_{\alpha ij}(x))_{1\le i\le r_2,1\le j\le r_1}
\end{equation*}
are $r_2\times r_1$ matrices with smooth coefficients,
$E_{1|\Omega}\simeq\Omega\times\C^{r_1}$, $E_{2|\Omega}\simeq\Omega\times\C^{r_2}$ are trivialized locally on some open chart $\Omega\subset M$ equipped with local coordinates $x^1,\dots,x^{m}$, 
\begin{equation*}
D^\alpha=(\del/\del x^1)^{\alpha_1}\dots(\del/\del x^{m})^{\alpha_{m}},
\end{equation*}
and $u=(u_j)_{1\le j\le r_1}$, $D^\alpha u=(D^\alpha u_j)_{1\le j\le r_1}$ are viewed as column vectors. Moreover, we require that $a_\alpha\nequiv0$ for some open chart $\Omega \subset M$ and some choice of multi-index $\alpha$ with $\lv\alpha\rv = l$.

Let $P:\Gamma(M,E_1)\to \Gamma(M,E_2)$ be a $\C$-linear differential operator of order $l$ from $E_1$ to $E_2$. The \emph{principal symbol}, or simply the \emph{symbol}, of $P$ is the map
\begin{equation*}
\sigma_P:T^*M\to \Hom(E_1,E_2)\ \ \ (x,\xi)\mapsto\sum_{\lv\alpha\rv= l}a_\alpha(x)\xi^\alpha.
\end{equation*}
We say that $P$ is \emph{elliptic} if $\sigma_P(x,\xi)\in\Hom((E_1)_x,(E_2)_x)$ is an isomorphism for all $x\in M$ and $0\ne\xi\in T_x^*M$. Note that, if $u\in\Gamma(M,E_1)$ and $f\in\cinf(M)$, with $f(x)=0$ then
\begin{equation*}\label{eq-sym}
P(f^lu)(x)=l!\sigma_P(x,df(x))(u(x)).
\end{equation*}
Therefore, we observe that $P$ is elliptic if and only if for all $x\in M$, $u\in\Gamma(M,E_1)$ and $f\in\cinf(M)$ such that $u(x)\ne0$, $f(x)=0$ and $df(x)\ne0$ we have
\begin{equation*}
P(f^lu)(x)\ne0.
\end{equation*}

Let $(M,g)$ be an oriented Riemannian manifold of dimension $m$. The metric $g$ induces the standard Riemannian volume form, given locally by
\begin{equation*}
\vol(x)=|\det g_{ij}(x)|^{\frac12}dx^1\wedge\dots\wedge dx^m,
\end{equation*}
where $g(x)=\sum g_{ij}(x)dx^i\otimes dx^j$ for local coordinates $x^1,\dots,x^m$. Likewise, we also obtain the standard Riemannian measure from $g$. 

Let $E$ be a $\C$-vector bundle over $M$, and take a Hermitian metric $h$ over $E$, \textit{i.e.}, a smooth section of Hermitian inner products on the fibers. The couple $(E,h)$ will be called a \emph{Hermitian vector bundle}. The separable Banach space $L^p E, p\ge 1$, is then given by (equivalence classes of almost everywhere equal) possibly non-continuous global sections $u$ of $E$, with measurable coefficients and finite $L^p$ norm. 
\begin{equation*}
\lV u\rV_{L^p}:=\left(\int_M| u(x)|^p\vol(x)\right)^{\frac1{p}}<+\infty,
\end{equation*}
where $|\cdot|=h(\cdot,\cdot)^\frac12$. The space $L^pE$ can be seen as the completion, with respect to the $L^p$ norm, of $\Gamma_0(M,E)$, the space of smooth sections with compact support.  
We denote by $L^p_{loc}E$ the space of global sections $u$ of $E$ such that $u \chi_K \in L^pE$ for every compact $K \subseteq M$, where $\chi_K = 1$ on $K$ and $\chi_K = 0$ otherwise.  
For $p=2$, we denote the corresponding $L^2$ inner product by 
\begin{equation*}
\la u,v\ra:=\int_M h(u(x),v(x))\vol(x).
\end{equation*}
The space $L^2E$ together with $\la\cdot,\cdot\ra$ is a separable Hilbert space. Denote by $\lv\cdot\rv$ the $L^2$ norm $\lv\cdot\rv_{L^2}$.

Let $(E_1,h_1),(E_2,h_2)$ be Hermitian vector bundles, and let $P:\Gamma(M,E_1)\to \Gamma(M,E_2)$ be a differential operator. The $L^2$ \emph{formal adjoint}, or simply the \emph{formal adjoint}, 
\begin{equation*}
P^*:\Gamma(M,E_2)\to \Gamma(M,E_1)
\end{equation*}
of $P$ is defined such that it satisfies the property 
\begin{equation*}
\la Pu,v\ra=\la u,P^*v\ra
\end{equation*}
for any pair of smooth sections $u\in\Gamma(M,E_1)$ and $v\in\Gamma(M,E_2)$ with $\supp u\cap\supp v$ compactly contained in $M$.
We remark that the $L^2$ formal adjoint $P^*$ is a differential operator, it always exists and it is unique (see, \textit{e.g.}, \cite[Chapter VI, Definition 1.5]{De}). It follows that $P^{**}=P$ and, if $Q$ is another differential operator, $(QP)^*=P^*Q^*$. An operator $P:\Gamma(M,E_1)\to \Gamma(M,E_1)$ is called \emph{formally self-adjoint} if $P=P^*$.

As a generalisation of the notion of an elliptic differential operator, we recall the definition of an elliptic complex. 
We start by considering a complex of differential operators. Take $E_j$ to be a sequence of $\C$-vector bundles over $M$ and take $D_j$ to be a sequence of $\C$-linear differential operators of order $k_j$
\begin{equation}\label{eq standard complex}
\dots\longrightarrow\Gamma(M,E_j)\overset{D_j}{\longrightarrow}\Gamma(M,E_{j+1})\overset{D_{j+1}}{\longrightarrow}\Gamma(M,E_{j+2}){\longrightarrow}\dots
\end{equation}
such that $D_j\circ D_{j+1}=0$ for all $j$. Given Hermitian metrics $h_j$ on $E_j$, this complex of differential operators has naturally associated, formally self-adjoint operators
\begin{equation}\label{eq defin non-ell lapl}
\Delta_j:=D_j^*D_j+D_{j-1}D_{j-1}^*:\Gamma(M,E_j)\longrightarrow\Gamma(M,E_{j+1})
\end{equation}
of order $2\max(k_j,k_{j-1})$ which are called Laplacians. 

We can view the principal symbol of the operator $D_j$ as a map
\begin{equation*}
\sigma_{D_j}:\pi^*E_j\to \pi^*E_{j+1},
\end{equation*}
where $\pi:T^*M\setminus\{\text{zero section}\}\to M$ and $\pi^*E_j$ denotes the pullback bundle of $E_j$.
 In this way, we say that \eqref{eq standard complex} is an \emph{elliptic complex} if the induced sequence of symbols
\[
\dots\longrightarrow\pi^*E_j\overset{\sigma_{D_j}}{\longrightarrow}\pi^*E_{j+1}\overset{\sigma_{D_{j+1}}}{\longrightarrow}\pi^*E_{j+2}{\longrightarrow}\dots
\]
is exact, \textit{i.e.}, if $\im\sigma_{D_j}=\ker\sigma_{D_{j+1}}$ for all $j$.  
Note that in general, even for elliptic complexes, $\Delta_j$ is not an elliptic operator.

There is, however, a different formally self-adjoint operator naturally associated to the elliptic complex, which is always elliptic (see \cite[Section I]{V}).
Focusing just on two consecutive maps in the complex
\[
\Gamma(M,E_1)\overset{D_1}{\longrightarrow}\Gamma(M,E_{2})\overset{D_{2}}{\longrightarrow}\Gamma(M,E_{3}),
\]
For $j=1,2$ we set
\[
l_j:=\frac{\lcm(k_1,k_2)}{k_j},\ \ \ r:=k_1l_1=k_2l_2=\lcm(k_1,k_2),
\]
and define the operator
\begin{equation}\label{eq defin ell lapl}
\square:=(D_1D_1^*)^{l_1}+(D_2^*D_2)^{l_2}
\end{equation}
of order $2r$, which is clearly formally self-adjoint.
To see that $\square$ is elliptic, one can argue similarly to \cite[Lemma 9.4.2]{Pa}.

Finally, we state the following result about elliptic regularity, whose proof follows, \textit{e.g.}, from \cite[Lemma 1.1.17]{L}.
Let $(E_1,h_1),(E_2,h_2)$ be Hermitian vector bundles, and let $P:\Gamma(M,E_1)\to \Gamma(M,E_2)$ be a differential operator.
We say that the section $u$ is a \emph{weak solution} of $Pu=v$ if $u\in L^1_{loc}E_1$, $v\in L^1_{loc}E_2$ and 
\begin{equation*}
\la u,P^*w\ra=\la v,w\ra\ \ \ \forall w\in\Gamma_0(M,E_2).
\end{equation*}
\begin{theorem}[Elliptic regularity]\label{thm ell-reg}
Given an oriented Riemannian manifold $(M,g)$, let $(E_1,h_1)$, $(E_2,h_2)$ be a pair of Hermitian vector bundles over $M$ and let $P:\Gamma(M,E_1)\to \Gamma(M,E_2)$ be an elliptic differential operator. If $u\in L^1_{loc}E_1$, $u$ is a weak solution of $Pu=v$ and $v$ is smooth, then $u$ must be smooth.
\end{theorem}

\section{Elliptic complexes on complex manifolds}\label{sec complex manifold}
Good references for the content of this section are \cite{De} and \cite{S}.
Let $M$ be a complex manifold of complex dimension $n$. We will use $\Lambda^r M$, $\Lambda^{r}_{\mathbb C}M$ and $\Lambda^{p,q} M$ to denote the bundles of real-valued $r$-forms, complex-valued $r$-forms, and $(p,q)$-forms, respectively. When the manifold is obvious from context, we will omit it and simply write $\Lambda^r$, 
 $\Lambda^{r}_{\mathbb C}$ and $\Lambda^{p,q}$. The spaces of smooth sections of these bundles are then given by
$$A^r(M) := \Gamma(M,\Lambda^rM), \quad  A^r_{\mathbb C}(M) := \Gamma(M,\Lambda^r_{\mathbb{C}} M), \quad A^{p,q}(M) := \Gamma(M,\Lambda^{p,q}M), $$
or simply $A^r$, $A^r_{\mathbb C}$ and $A^{p,q}$.

The exterior derivative on $A^{p,q}$ splits into two bidegree components
\[
d=\del+\delbar,
\]
where \[
\del:A^{p,q}\rightarrow A^{p+1,q},\quad \delbar:A^{p,q}\rightarrow  A^{p,q+1},
\] 
therefore the relation $d^2=0$ immediately implies
\[
\del^2=0,\ \ \ \delbar^2=0,\ \ \ \del\delbar+\delbar\del=0.
\]
Thanks to these relations there are complexes of differential forms which are naturally associated to a complex manifold.

The first complex we are interested in is the Dolbeault complex, given by
\[
\dots\longrightarrow A^{p,q-1}\overset{\delbar}{\longrightarrow} A^{p,q}\overset{\delbar}{\longrightarrow} A^{p,q+1}{\longrightarrow}\dots
\]
for all $0\le p,q\le n$, along with 
the associated Dolbeault cohomology
\[
H^{p,q}_\delbar:=\frac{\ker\delbar\cap A^{p,q}}{\im\delbar\cap A^{p,q}}.
\]
There is an analogous complex for the map $\del$, with the associated cohomology $H^{p,q}_\del$.

The second complex we are interested in is the Aeppli-Bott-Chern complex, given by
\begin{equation}\label{eq bc complex}
\begin{tikzcd}
\dots\arrow[d]\\
 A^{p-1,q-2}\oplus A^{p-2,q-1}\arrow[d,"\delbar \oplus \del"]\\
 A^{p-1,q-1}\arrow[d,"\del\delbar"]\\
 A^{p,q}\arrow[d,"\del+\delbar"]\\
 A^{p+1,q}\oplus A^{p,q+1}\arrow[d]\\
\dots
\end{tikzcd}
\end{equation}
for all $0\le p,q\le n$. The complete definition of the ABC complex is given in Section \ref{sec questions} (\textit{cf}. \cite[Chapter VI, Section 12.1]{De}). The Bott-Chern and Aeppli cohomologies
\[
H^{p,q}_{BC}:=\frac{\ker \del\cap\ker\delbar\cap A^{p,q}}{\im\del\delbar\cap A^{p,q}},\ \ \ H^{p-1,q-1}_A:=\frac{\ker \del\delbar\cap A^{p-1,q-1}}{\im\delbar\oplus\del\cap A^{p-1,q-1}}
\]
can be obtained from this complex.
 
There are well-defined maps between these cohomology spaces, induced by the identity on the representatives of the cohomology classes. Specifically we have the following commutative diagram of maps
\begin{equation}\label{diagram maps cohomology}
\begin{tikzcd}
{}&{H^{\bullet,\bullet}_{BC}}\arrow[dl]\arrow[d]\arrow[dr] &{}\\
{H^{\bullet,\bullet}_\del}\arrow[dr]&{H^\bullet_{dR}}\arrow[d] &{H^{\bullet,\bullet}_\delbar}\arrow[dl]\\
{}&{H^{\bullet,\bullet}_A}&{}
\end{tikzcd}
\end{equation}
where $H^k_{dR}$ is the de Rham cohomology.
Moreover, by \cite[Lemma 5.15, Remark 5.16]{DGMS}, these maps are all isomorphisms if and only if one of the following equivalent conditions holds on $A^\bullet_\C$:
\begin{enumerate}[a)]
\item $\im\del\delbar=\ker\del\cap\ker\delbar\cap\im d$;
\item $\im\del\delbar=\ker\del\cap\im\delbar$;
\item $\im\del\delbar=\ker\del\cap\ker\delbar\cap(\im\del+\im\delbar)$;
\item $\ker\del\delbar=\im\del+\im\delbar+\ker d$;
\item $\ker\del\delbar=\im\delbar+\ker\del$;
\item $\ker\del\delbar=\im\del+\im\delbar+(\ker\del\cap\ker\delbar)$.
\end{enumerate}
Note that when condition a) is satisfied it is often said that the $\del\delbar$-Lemma holds.

Using linear algebra one can show that the Dolbeault and the Aeppli-Bott-Chern complexes are both elliptic, see \cite[Section II]{V} or \cite[Lemma 2]{Ste}. Therefore, once we fix a Hermitian metric for differential forms, there are naturally associated elliptic and formally self-adjoint operators as defined in \eqref{eq defin ell lapl}.

Let $(M,g)$ be a Hermitian manifold of complex dimension $n$, that is a complex manifold $M$ endowed with a Hermitian metric $g$. Recall that a Hermitian metric on a complex manifold is a Riemannian metric $g$ for which the complex structure $J$ is an isometry, \textit{i.e.}, $g(J\cdot,J\cdot)=g(\cdot,\cdot)$. We will generally denote by $\omega$ the fundamental $(1,1)$-form associated to the metric $g$, which is defined by $\omega(\cdot,\cdot):=g(J\cdot,\cdot)$.  We will also typically denote by $h$ the Hermitian extension of $g$ on the complexified tangent bundle $T^\C M=TM\otimes\C=T^{1,0}M\oplus T^{0,1}M$, and by the same symbol $g$ the $\C$-bilinear symmetric extension of $g$ on $T^\C M$. Consequently, we have $h(u,v)=g(u,\bar{v})$ for all $u,v\in \Gamma(M,T^{1,0}M)$. Note that the standard Riemannian volume form can be computed by $\vol=\frac{\omega^n}{n!}$.


The Hermitian metric $g$ extends to a Hermitian metric $h$ on $\Lambda^{p,q}M$, defined pointwise as follows. At any point $x \in M$, we choose a basis $\{v_1,\dots,v_n\}$ of $T_x^{1,0}M$, such that $h_x(v_i,{v}_j)=\delta_{ij}$. Take the dual basis $\{\alpha_1,\dots,\alpha_n\}$ of $\Lambda_x^{1,0}$ and define $h_x$ on $\Lambda_x^{1,0}$ in such a way that $h_x(\alpha_i,{\alpha}_j)=\delta_{ij}$. We then extend the Hermitian metric to $\Lambda_x^{p,q}$ by setting
\begin{equation*}\label{eq-metric-forms}
h_x(\alpha^{i_1}\wedge\dots\wedge\c\alpha^{j_q},\alpha^{k_1}\wedge\dots\wedge\c\alpha^{h_q})=\delta_{i_1 k_1}\dots\delta_{j_q h_q}.
\end{equation*}
Then, $L^2 \Lambda^{p,q}$ is the space of (equivalence classes of almost everywhere equal) measurable $(p,q)$-forms $\phi$, such that
\begin{equation*}
\lVert \phi\rVert:=\Big(\int_M h(\phi,\phi) \vol\Big)^\frac12<\infty.
\end{equation*}
Pairing $L^2 \Lambda^{p,q}$ with the Hermitian inner product
\begin{equation*}
\la\varphi,\psi\ra :=\int_M h(\varphi,\psi) \vol,
\end{equation*}
we get a Hilbert space. The space $L^2 \Lambda^{p,q}$ can be also seen as the completion of $A^{p,q}_0:= \Gamma_0(M,\Lambda^{p,q})$, the space of smooth $(p,q)$-forms with compact support, with respect to the norm $\lVert\cdot\rVert$.

The complex $\C$-linear Hodge operator $*: A^{p,q}\to  A^{n-q,n-p}$ associated with the metric is defined by the equation
\[
\alpha\wedge*\c\beta=h(\alpha,\beta)\vol
\]
for all $\alpha,\beta\in A^{p,q}$. We set also
\begin{equation*}
\delbar^*:=-*\del*\ \ \ \del^*:=-*\delbar *,\ \ \ d^*:=-*d*,
\end{equation*}
which are the $L^2$-formal adjoints respectively of $\del,\delbar,d$ by a direct application of the Stokes Theorem.

Note that the Laplacians defined in \eqref{eq defin non-ell lapl} and \eqref{eq defin ell lapl} coincide, and so there is a single elliptic and formally self-adjoint Laplacian associated to the Dolbeault complex,
\[
\Delta_{\delbar}:=\delbar\delbar^*+\delbar^*\delbar,
\]
known as the {\em Dolbeault Laplacian}. 
Similarly we define
\begin{equation*}
\Delta_{\del}:=\del\del^*+\del^*\del,\ \ \ \Delta_{d}:=dd^*+d^*d,
\end{equation*}
which are respectively called {\em $\del$-Laplacian} and the {\em Hodge Laplacian}.

For the Aeppli-Bott-Chern complex we obtain different Laplacians from \eqref{eq defin non-ell lapl} and \eqref{eq defin ell lapl}. Note that the formal adjoint of the map $\del+\delbar$ in \eqref{eq bc complex} is $\del^*\oplus\delbar^*$, while the formal adjoint of $\delbar\oplus\del$ in \eqref{eq bc complex} is $\del^*+\delbar^*$.
 The operators
\[
\Delta_{BC}:=\del\delbar\delbar^*\del^*+(\del^*\oplus\delbar^*)(\del+\delbar)=\del\delbar\delbar^*\del^*+\del^*\del+\delbar^*\delbar,
\]
\[
\Delta_A:=\delbar^*\del^*\del\delbar+(\delbar\oplus\del)(\del^*+\delbar^*)=\delbar^*\del^*\del\delbar+\del\del^*+\delbar\delbar^*,
\]
acting respectively on $A^{p,q}$ and $A^{p-1,q-1}$,
are the Laplacians given by \eqref{eq defin non-ell lapl}. Recall that they are formally self-adjoint but, as noted in \cite{S}, they are not elliptic.
Whereas, the operators
\[
\square_{BC}:=\del\delbar\delbar^*\del^*+((\del^*\oplus\delbar^*)(\del+\delbar))^2=\del\delbar\delbar^*\del^*+(\del^*\del+\delbar^*\delbar)^2,
\]
\[
\square_{A}:=\delbar^*\del^*\del\delbar+((\delbar\oplus\del)(\del^*+\delbar^*))^2=\delbar^*\del^*\del\delbar+(\del\del^*+\delbar\delbar^*)^2,
\]
acting respectively on $A^{p,q}$ and $A^{p-1,q-1}$, are the Laplacians given by \eqref{eq defin ell lapl} and are both formally self-adjoint and elliptic \cite[Section 2]{V}.

Moreover, Kodaira and Spencer \cite{KS} have introduced the following fourth order elliptic and formally self-adjoint differential operators usually referred to as the {\em Bott-Chern} and {\em Aeppli Laplacians}, which are defined respectively as
\begin{equation*}
\tilde\Delta_{BC} :=
\del\delbar\delbar^*\del^*+
\delbar^*\del^*\del\delbar+\del^*\delbar\delbar^*\del+\delbar^*\del\del^*\delbar
+\del^*\del+\delbar^*\delbar,
\end{equation*}
\begin{equation*}
\tilde\Delta_{A} :=
\del\delbar\delbar^*\del^*+
\delbar^*\del^*\del\delbar+
\del\delbar^*\delbar\del^*+\delbar\del^*\del\delbar^*+
\del\del^*+\delbar\delbar^*.
\end{equation*}

All these Laplacians are linked by the following duality relations
\begin{align}
 *\Delta_{A}&=\Delta_{BC}*,& *\Delta_{BC}&=\Delta_{A}*,\notag\\
 *\tilde\Delta_{A}&=\tilde\Delta_{BC}*,& *\tilde\Delta_{BC}&=\tilde\Delta_{A}*,\label{eq duality bc a}\\
 *\square_{A}&=\square_{BC}*,& *\square_{BC}&=\square_{A}*.\notag
\end{align}

On a compact complex manifold, it is straightforward to see that
\[
\ker \Delta_\delbar=\ker\delbar\cap\ker\delbar^*,\ \ \ \ker \Delta_\del=\ker\del\cap\ker\del^*,\ \ \  \ker \Delta_d=\ker d\cap\ker d^*,
\]
\[
\ker \Delta_{BC}=\ker \tilde\Delta_{BC}=\ker \square_{BC}=\ker\delbar^*\del^*\cap\ker \del\cap\ker\delbar,
\]
\[
\ker \Delta_{A}=\ker \tilde\Delta_{A}=\ker \square_{A}=\ker\del\delbar\cap\ker \del^*\cap\ker\delbar^*.
\]

Notice that the kernels  of the Bott-Chern Laplacians $\Delta_{BC}, \tilde{\Delta}_{BC}$ and $\square_{BC}$ all coincide, and likewise for the Aeppli Laplacian.
The respective spaces of harmonic forms will be denoted by
\[
\H^k_{d}:=\ker\Delta_d\cap A^k,\ \ \ \H^{p,q}_{\delta}:=\ker\Delta_\delta\cap A^{p,q},
\]
for $\delta\in\{\del,\delbar,BC,A\}$, as a consequence of Hodge theory they are finite dimensional and isomorphic to the respective cohomology spaces
\[
\H^k_d\simeq H^{k}_{dR},\ \ \ \H^{p,q}_\delta\simeq H^{p,q}_\delta.
\]
Their dimensions will be indicated by
\[
b^k:=\dim_\R H^{k}_{dR},\ \ \ h^{p,q}_{\delta}:=\dim_\C H^{p,q}_\delta.
\]

If $(M,g)$ is a K\"ahler manifold (\textit{i.e.}, $d \omega = 0$, where $\omega$ is the fundamental form), it is well known that the second order Laplacians coincide up to a factor,
\begin{equation}
    \label{eq laplacians 2order kahler}
\Delta_{d}=2\Delta_{\del}=2\Delta_{\delbar}.
\end{equation}
Furthermore, using the K\"ahler identities, see, \textit{e.g.}, \cite[Chapter VI, Theorem 6.4]{De}, we know that $\del$ and $\delbar^*$ anticommute, as do $\del^*$ and $\delbar$. This allows us to write $\tilde\Delta_{BC}$ and $\tilde\Delta_{A}$ in a more concise form,
\begin{gather}
\label{eq laplacian bc kahler}
\tilde\Delta_{BC}=\Delta_{\delbar}\Delta_{\delbar}+
\del^*\del+\delbar^*\delbar,\\
\label{eq laplacian a kahler}
\tilde\Delta_{A}= \Delta_{\delbar}\Delta_{\delbar}+
\del\del^*+\delbar\delbar^*.
\end{gather}
As a consequence, all the spaces of harmonic forms coincide
\[
\H^k_d\otimes\C\cap A^{p,q}=\H^{p,q}_{\del}=\H^{p,q}_{\delbar}=\H^{p,q}_{BC}=\H^{p,q}_{A}.
\]
In particular, this implies that the $\del\delbar$-Lemma holds, see \cite[Lemma 5.11, Remark 5.14]{DGMS}.

\section{Unbounded operators on Hilbert spaces}\label{sec hilb}
We recall some concepts from the theory of unbounded operators on Hilbert spaces. For a complete treatment we refer to \cite[Chapter 1]{F} or \cite[Chapter 8]{RS1}. 
 If $\H_1,\H_2$ are Hilbert spaces, a \emph{linear operator} $P:\H_1\to\H_2$ is a linear function $P:\D(P)\to\H_2$, defined on a \emph{domain} $\D(P)\subseteq\H_1$.
The \emph{graph} of a linear operator $P:\H_1\to\H_2$ is the subspace of $\mathcal{H}_1 \times \mathcal{H}_2$ given by $\gr(P):=\{(x,Px)\in\H_1\times\H_2\,|\,x\in\D(P)\}$. A linear operator is \emph{closed} if its graph is a closed subspace. 
By the closed graph theorem, a closed linear operator which is defined everywhere on $\mathcal{H}_1$ is automatically bounded, therefore when dealing with an unbounded operator we must always keep track of its domain.

The \emph{kernel} of a linear operator $P:\H_1\to\H_2$ with domain $\D(P)$ is the subspace $\ker P:=\{x\in\D(P)\,|\,Px=0\}$, while its \emph{image} is $\im P:=P(\D(P))$. If $P$ is closed, then its kernel is a closed subspace.

An \emph{extension} of $P$ is a linear operator $P'$ such that $\D(P)\subseteq\D(P')$ and $Px=P'x$ for every $x\in\D(P)$, \textit{i.e.}, $\gr(P)\subseteq\gr(P')$; in this case we will write $P\subseteq P'$. A linear operator $P$ is \emph{closable} if it admits a closed extension. Every closable operator $P$ has a smallest closed extension, denoted by $\c{P}$, which is called its \emph{closure}, and is given by the intersection of the graphs of all closed extensions. Note that a linear subspace $G$ of $\H_1\times\H_2$ is a graph of a linear operator so long as $(0,y) \notin G$ for all non-zero $y\in \mathcal{H}_2$. If $P$ is closable then $\gr(\c{P})=\c{\gr(P)}$, \cite[p. 250]{RS1}.

If $\D(P)$ is dense in $\H_1$, then we say that the linear operator $P$ is \emph{densely defined}, and we can define the \emph{Hilbert adjoint}, or simply the \emph{adjoint}, of $P$, indicated by $P^t:\H_2\to\H_1$. Its domain is
\begin{equation*}
\D(P^t):=\{y\in\H_2\,|\,x\mapsto\la Px,y\ra_2\text{ is bounded on }\D(P)\},
\end{equation*}
where $\la\cdot,\cdot\ra_i$ denotes the Hermitian inner product of the Hilbert space $\H_i$.
The adjoint is then defined by the relation
\begin{equation*}
\la Px,y\ra_2=\la x,P^ty\ra_1\ \forall x\in\D(P)\  \forall y\in\D(P^t).
\end{equation*}
Indeed, if $y \in\D(P^t)$, then the map $x\mapsto\la Px,y\ra_2$ can be uniquely extended from a densely defined function on $\D(P)$ to a bounded linear function on all of $\mathcal{H}_1$. By the Riesz representation theorem, there exists a unique $z \in \H_1$ such that $\la Px,y\ra_2=\la x,z\ra_1$, for all $x \in \D(P)$. We then define $P^t y := z$.

We can verify that $\gr(P^t)=(\gr(-P))^\perp$ in $\H_1\times\H_2$, thus this definition makes $P^t$ a closed linear operator. If $P$ is closed, this implies that every pair $(u,v)\in\H_1\times\H_2$ can be written as the sum of elements in $\gr(P^t)$ and $\gr(-P)$,
\[
(u,v)=(x,-Px)+(P^ty,y),\ \ \ x\in\D(P),\ y\in\D(P^t).
\]
If $u=0$, then we have
\[
x+P^ty=0,\ \ \ v=y-Px=y+PP^ty,\ \ \ \la v,y\ra_2=\lv y\rv_2^2+\lv P^ty\rv_1^2,
\]
where $\lv\cdot\rv_i:=\la\cdot,\cdot\ra_i^{\frac12}$. If $v\in\D(P^t)^\perp$, we get $\la v,y\ra_2=0$, thus $y=0$ and $v=0$. Therefore, $P^t$ is densely defined and so we can define $P^{tt}$. In fact, we can verify that $P^{tt}=P$. 

The above discussion (see also \cite[Theorem VIII.1]{RS1}) implies the following result.
\begin{lemma}
\label{lem dem adjoint ker im}
Let $\H_1,\H_2$ be two complex Hilbert spaces and $P:\H_1\to\H_2$ be a linear densely defined operator. If $P$ is closed, then its adjoint $P^t$ is closed and densely defined, $P^{tt}=P$ and 
\[
\ker P^t=\im P^\perp\ \ \ \ \ \ker P^\perp=\c{\im P^t}.
 \]
Moreover $P$ is closable iff $\D(P^t)$ is dense, in which case $(\c{P})^t=P^t$ and $\c{P}=P^{tt}$.
\end{lemma}

Given two densely defined linear operators $P,Q$, it holds that
\begin{equation*}\label{eq extension adjoint}
P\subseteq Q\ \ \ \iff\ \ \ Q^t\subseteq P^t,
\end{equation*}
\begin{equation*}\label{eq inclusion im ker adjoint}
\im P\subseteq \ker Q\ \ \ \iff\ \ \ \im Q^t\subseteq \ker P^t.
\end{equation*}
It follows that any Hilbert complex
\[
\H_1\overset{P}{\longrightarrow}\H_2\overset{Q}{\longrightarrow}\H_3,
\]
\textit{i.e.}, any sequence of closed and densely defined linear operators $P,Q$, such that $\im P\subseteq \ker Q$,
has an associated \emph{dual Hilbert complex}
\[
\H_1\overset{P^t}{\longleftarrow}\H_2\overset{Q^t}{\longleftarrow}\H_3.
\]
Another fundamental property of Hilbert complexes is the following orthogonal decomposition.

\begin{theorem}
\label{thm decomp hilb dem}
Let $P:\H_1\to\H_2$ and $Q:\H_2\to\H_3$ be closed and densely defined linear operators between Hilbert spaces. If $\im P\subseteq\ker Q$, then we have the orthogonal decompositions
\begin{align*}
\H_2&=\ker Q\cap\ker P^t\oplus\c{\im P}\oplus\c{\im Q^t},\\
\ker Q&=\ker Q\cap\ker P^t\oplus\c{\im P},\\
\ker P^t &= \ker Q \cap \ker P^t \oplus \c{\im Q^t}.
\end{align*}
\end{theorem}
\begin{proof}
By Lemma \ref{lem dem adjoint ker im}, the Hilbert space $\H_2$ has an orthogonal decomposition
\[
\H_2=\ker Q\oplus\c{\im Q^t},
\]
and analogously
\[
\H_2=\ker P^t\oplus\c{\im P}.
\]
Furthermore, since $\ker Q$ is a closed subspace of $\mathcal{H}_2$, it is itself a Hilbert space with $\c{\im P} \subseteq \ker Q$. Again by Lemma \ref{lem dem adjoint ker im}, we see that $\ker Q$ has an orthogonal decomposition
\[
\ker Q=\ker Q\cap\ker P^t\oplus\c{\im P}.
\]
The last decomposition follows similarly.
\end{proof}

The following basic lemma characterises the closure of the image of a closed and densely defined operator. We include a proof for the convenience of the reader.
\begin{lemma}\label{lemma im closed}
Let $P:\H_1\to\H_2$ be a closed and densely defined linear operator between Hilbert spaces. The following conditions are equivalent:
\begin{enumerate}[label=\upshape{\alph*)}]
\item $\im P$ is closed;
\item $\exists C>0 \text{ s.t. } \lv x\rv_1\le C\lv Px\rv_2$ for all $x\in\D(P)\cap \c{\im P^t}$;
\item $\im P^t$ is closed;
\item $\exists C>0 \text{ s.t. } \lv y\rv_2\le C\lv P^ty\rv_1$ for all $y\in\D(P^t)\cap \c{\im P}$.
\end{enumerate}
\end{lemma}
\begin{proof}
By Lemma \ref{lem dem adjoint ker im}, the Hilbert space $\H_1$ can be orthogonally decomposed as 
\begin{equation}\label{eq orth decomp P}
\H_1=\ker P\oplus \c{\im P^t},
\end{equation}
which implies that the operator $P_{|\D(P)\cap\c{\im P^t}}:\D(P)\cap\c{\im P^t}\to\im P$ is bijective, therefore its inverse is well-defined and one can check that it is also a closed operator. If a) holds, then by the closed graph theorem $(P_{|\D(P)\cap\c{\im P^t}})^{-1}$ is continuous, proving b). Conversely, if b) holds, then using \eqref{eq orth decomp P} we can show that any Cauchy sequence in $\im P$ converges in $\im P$, and thus we obtain a). Similarly c) is equivalent to d). To prove that b) implies d), note that by b)
\[
|\la y, Px\ra_2|=|\la P^ty, x\ra_1|\le C\lv P^ty\rv_1\lv Px\rv_2
\]
for all $y\in\D(P^t)$ and $x\in\D(P)\cap \c{\im P^t}$. Therefore
\[
|\la y, z\ra_2|\le C\lv P^ty\rv_1\lv z\rv_2
\]
for all $y\in\D(P^t)$ and $z\in\c{\im P}$, again using \eqref{eq orth decomp P}. Choosing $z=y$ then implies d). Since $P^{tt}=P$, a similar argument shows that d) implies b).
\end{proof}

Let $\H$ be a Hilbert space.
A densely defined linear operator $P:\H\to\H$ is called \emph{symmetric} if $\la Px,y\ra=\la x,Py\ra$ whenever $x,y\in\D(P)$, or equivalently if $P\subseteq P^t$. It is called \emph{self-adjoint} if it is symmetric and $\D(P)=\D(P^t)$, or equivalently if $P=P^t$ (with equality of domains). A symmetric operator $P$ is always closable since its adjoint is a closed extension, therefore $\c{P}=P^{tt}$ and so $P\subseteq P^{tt}\subseteq P^t$. A linear operator $P$ is  called \emph{essentially self-adjoint} if it has a unique self-adjoint extension. Equivalently, $P$ is essentially self-adjoint if $\c{P}$ is a self-adjoint operator, \cite[p. 256]{RS1}.  An operator is called \emph{positive} if $\la Px,x\ra\ge0$ whenever $x\in\D(P)$.

The following Theorems describe a method for building self-adjoint operators.

\begin{theorem}[{\cite[Theorem X.25]{RS2}} or {\cite[Theorem 2.3]{F}}]\label{thm composition with adjoint}
Let $P$ be a closed and densely defined linear operator on a Hilbert space $\H$. Then the operator $P^tP$ defined by $(P^tP)x=P^t(Px)$ on the domain
\[
\D(P^t P):=\{x\in\D(P)\,|\,Px\in\D(P^t)\}
\]
is positive and self-adjoint.
\end{theorem}

\begin{theorem}[{\cite[Theorem 4.1]{F}}]\label{thm sum of self ajoint}
Let $P,Q$ be positive and self-adjoint operators on a Hilbert space $\H$. Assume that $\D(P)\cap \D(Q)$ is dense in $\H$. Then the operator $P+Q$ defined by $(P+Q)x=Px+Qx$ on the domain
\[
\D(P+Q):=\D(P)\cap \D(Q)
\]
is positive and self-adjoint.
\end{theorem}

By Theorems \ref{thm composition with adjoint} and \ref{thm sum of self ajoint}, any Hilbert complex
\[
\H_1\overset{P}{\longrightarrow}\H_2\overset{Q}{\longrightarrow}\H_3
\]
has a naturally associated Laplacian, given by
\begin{equation}\label{Hilbert complex laplacian}
    \Delta:=PP^t+Q^tQ,
\end{equation}
which is positive and self-adjoint. Note that the Laplacian of a Hilbert complex coincides with the Laplacian of its dual Hilbert complex.
We will need later on the following observation: being $\Delta$ self adjoint, we can define $\Delta^2$ by Theorem \ref{thm composition with adjoint}, and its domain turns out to satisfy
\begin{equation}\label{eq domain square lapl}
\D(\Delta^2)=\D((PP^t)^2)\cap\D((Q^tQ)^2).
\end{equation}
Finally, it is easy to check that the kernel of the Laplacian is characterised by
\[
\ker\Delta=\ker Q\cap \ker P^t.
\]
More generally, we have the following result.
\begin{lemma}\label{lemma ker lapl}
Let $P_j:\H_1\to\H_2$ be closed and densely defined linear operators for $j=1,\dots,n$, $n\in\N$, and define $\square:=\sum_{j=1}^nP_j^tP_j$ with the domain given by Theorems \ref{thm composition with adjoint} and \ref{thm sum of self ajoint} (we ask that $\bigcap_{j=1}^n\D(P_j^tP_j)$ is dense in $\H_1$). Then
\[
\ker\square=\bigcap_{j=1}^n\ker P_j.
\]
\end{lemma}

\subsection{Spectral gap of the Laplacian of a Hilbert complex}\,\\
We end this section by analysing some characterisations of the spectral gap condition of any positive self-adjoint operator and in particular of the Laplacian of a given Hilbert complex.

We say that $\lambda\in \mathbb{C}$ belongs to the \emph{spectrum} $\sigma(P)$ of an unbounded linear operator $P:\H\to\H$ if there exists no bounded linear operator $B:\H\to\H$ such that 
\begin{itemize}
    \item[1)] $B(P-\lambda I)x = x $ for all $x \in \mathcal{D}(P)$,
    \item[2)] $Bx\in \mathcal{D}(P)$ and $(P-\lambda I)Bx = x$ for all $x \in \mathcal{H}$;
\end{itemize}
in other words, if $P-\lambda I$ has no bounded inverse. By \cite[Theorem X.1]{RS2}, the spectrum of a self-adjoint operator is a subset of the real axis, and the spectrum of a positive self-adjoint operator is a subset of the non-negative real axis. Furthermore, we say that a positive self-adjoint operator $P$ has a \emph{spectral gap} if it has the property $\inf(\sigma (P) \backslash \{0\})=C>0$, or equivalently if $\sigma(P)\subseteq\{0\}\cup[C,+\infty)$ with $C>0$. 

We now state the Spectral Theorem for unbounded self-adjoint operators. We will introduce very briefly the notion of a direct integral.

Let $(X,\mu)$ be a $\sigma$-finite measure space and $\{\mathbf{H}_{\lambda}\}_{\lambda\in X}$ be a collection of separable Hilbert spaces and a \emph{measurability structure}, see \cite[Definition 7.18]{Ha} for more details. Denote by $\la\cdot,\cdot\ra_\lambda$ and $\lv\cdot\rv_\lambda$ the inner product and the norm on $\mathbf{H}_{\lambda}$.
A \emph{section} $s$ is a function $X\to\bigcup_{\lambda\in X}\mathbf{H}_{\lambda}$ satisfying $s(\lambda)\in\mathbf{H}_{\lambda}$ for all $\lambda\in X$.
The \emph{direct integral}
\[
\int_{X}^{\oplus}\mathbf{H}_{\lambda} d\mu(\lambda)
\]
is the Hilbert space of classes of a.e. equal \emph{measurable} sections $s$ with finite norm $\lv s\rv<+\infty$, where $\lv s\rv:=\la s,s\ra^{\frac12}$ and
\[
\la s_1,s_2\ra:=\int_X\la s_1(\lambda),s_2(\lambda)\ra_\lambda d\mu(\lambda).
\]

\begin{theorem}[Spectral Theorem]\label{thm spectral unbounded self-adjoint}
Let $P$ be a self-adjoint operator on a separable Hilbert space $\mathcal{H}$. Then there is a $\sigma$-finite measure $\mu$ on the spectrum $\sigma(P)$, along with a unitary map (a linear bijection preserving the inner products)
\[
U: \mathcal{H} \rightarrow \int_{\sigma(P)}^{\oplus}\mathbf{H}_{\lambda} d\mu(\lambda)
\]
such that
\[
U(\D(P)):=\left\{s\in\int_{\sigma(P)}^{\oplus}\mathbf{H}_{\lambda} d\mu(\lambda)\,:\,\int_{\sigma(P)}\lv\lambda s(\lambda)\rv_\lambda^2d\mu(\lambda)<+\infty  \right\}
\]
and, for all $s\in U(\D(P))$ and $\lambda\in\sigma(P)$,
\begin{equation}\label{eq spectral theorem}
(U P U^{-1}s)(\lambda) = \lambda s(\lambda).
\end{equation}
\end{theorem}
In this case, we say that $P$ is isomorphic to the multiplication (operator) by $\lambda$, \textit{i.e.}, the identity function on $\sigma(P)$.
We refer to \cite[Theorem 10.09, Section 10.4]{Ha} for the proof.
The point of the Spectral Theorem is that any question about a single self-adjoint operator is a question about a function. For example, we can apply measurable functions to self-adjoint operators as follows.
If $\phi$ is a complex Borel measurable function on the reals, then $\phi(P)$ is defined as the operator which is isomorphic (by the same unitary operator) to multiplication by $\phi(\lambda)$. \textit{E.g.}, the integer powers $P^k$ are well-defined for $k\in\Z$ and this definition of $P^2$ coincides with the one given by applying Theorem \ref{thm composition with adjoint}.

With the help of the Spectral Theorem, we can prove the following characterisation of the spectral gap property.
\begin{lemma}\label{lemma spectral gap equiv}
Let $P$ be a positive self-adjoint operator on a separable Hilbert space $\H$. Then the following conditions are equivalent:
\begin{enumerate}[label=\upshape{\alph*)}]
\item $P$ has a spectral gap;
\item $\exists C>0\ \la x,Px\ra\ge C\la x,x\ra$ for all $x\in\D(P)\cap(\ker P)^\perp$;
\item $\im P$ is closed.
\end{enumerate}
\end{lemma}
\begin{proof}
By Lemma \ref{lem dem adjoint ker im} and Lemma \ref{lemma im closed} it is immediate to see b) is equivalent to c).
We now prove that a) is equivalent to b).
For all $x\in\D(P)$, by the Spectral Theorem let $Ux=s$, so that
\begin{align*}
    \la x, Px \ra &= \int_{\sigma(P)}\langle s(\lambda), \lambda s(\lambda) \rangle_{\lambda} d\mu(\lambda) 
    =\int_{\sigma(P)\setminus\{0\}}\lambda \lv s(\lambda) \rv^2_\lambda d\mu(\lambda).
\end{align*}
Notice that $y\in\ker P$ iff, given $r=Uy$, $\lambda r(\lambda)=0$ a.e., iff $r(\lambda)=0$ a.e. for $\lambda\ne0$.
If $x\in\D(P)\cap(\ker P)^\perp$, then $s(0)=0$, so that
\[
\la x, x\ra=\int_{\sigma(P)}\langle s(\lambda), s(\lambda) \rangle_{\lambda} d\mu(\lambda) 
    =\int_{\sigma(P)\setminus\{0\}} \lv s(\lambda) \rv^2_\lambda d\mu(\lambda)
\]
Therefore $\la x,Px\ra\ge C\la x,x\ra$ for all $x\in\D(P)\cap(\ker P)^\perp$ iff $\lambda\ge C$ a.e. in $\sigma(P)\setminus\{0\}$, iff  $\sigma(P)\subseteq\{0\}\cup[C,+\infty)$.
\end{proof}

If the positive self-adjoint operator is the Laplacian associated to a Hilbert complex, then we obtain other characterisations of the spectral gap condition. 

\begin{lemma}\label{lemma spectral gap equiv 2}
Let $P:\H_1\to\H_2$ and $Q:\H_2\to\H_3$ be closed and densely defined linear operators between separable Hilbert spaces satisfying $\im P\subseteq\ker Q$. Define the positive self-adjoint operator
\[
\Delta=Q^tQ+PP^t,
\]
with the domain given by Theorems \ref{thm composition with adjoint} and \ref{thm sum of self ajoint} (we ask that $\D(Q^tQ)\cap\D(PP^t)$ is dense in $\H_2$).
Then the following conditions are equivalent:
\begin{enumerate}[label=\upshape{\alph*)}]
\item $\Delta$ has a spectral gap;
\item $\exists C>0\ C\lv x\rv^2_2 \le \lv P^tx\rv^2_1+\lv Qx\rv^2_3$ for all $x\in\D(P^t)\cap\D(Q)\cap(\ker \Delta)^\perp$;
\item $\im Q^t$ and $\im P$ are closed.
\end{enumerate}
\end{lemma}
\begin{proof}
Since $\ker\Delta=\ker P^t\cap \ker Q$, then $(\ker \Delta)^\perp=\c{\im P}\oplus\c{\im Q^t}$ by Lemma \ref{lem dem adjoint ker im}.
Therefore c) implies b) applying Lemma \ref{lemma im closed}.
Moreover, b) implies a) thanks to the second characterisation in Lemma \ref{lemma spectral gap equiv}. 
To prove that a) implies c), by Lemma \ref{lemma spectral gap equiv} we know that $\im \Delta$ is closed, therefore by Lemma \ref{lem dem adjoint ker im} we can decompose
\[
\H_2=\ker\Delta\oplus\im\Delta\subseteq\ker\Delta\oplus{\im P}\oplus{\im Q^t}\subseteq\H_2,
\]
thus
\[
\H_2=\ker\Delta\oplus{\im P}\oplus{\im Q^t}.
\]
The orthogonality of the previous decompositions follows from Lemma \ref{lem dem adjoint ker im} and by the assumption $\im Q\subseteq\ker P$.
The closure of $\im P$ and $\im Q^t$ follows easily from the orthogonality of the last decomposition.
\end{proof}

\section{Strong and weak extensions}\label{sec extensions}
Let $(M,g)$ be an oriented Riemannian manifold, let $(E_1,h_1)$ and $(E_2,h_2)$ be Hermitian vector bundles on $M$, and let ${P}:\Gamma(M,E_1)\to \Gamma(M,E_2)$ be a differential operator.  

Using $\Gamma_0(M,E_j)$ to denote the space of smooth sections of $E_j$ with compact support, we will write the restriction of $P$ to compactly supported sections as $P_0:\Gamma_0(M,E_1) \to \Gamma_0(M,E_2)$. This can be viewed as an unbounded, densely defined and closable linear operator $P_0:L^2E_1\to L^2E_2$, with domain $\Gamma_0(M,E_1)$. Below, we construct two canonical closed extensions of $P_0$, thereby verifying that $P_0$ is closable. 

The \emph{strong extension} $P_{s}$ (also called the \emph{minimal closed extension} $P_{min}$) is defined by taking the closure of the graph of ${P_0}$, \textit{i.e.}, $P_s = \c{P_0}$, or more explicitly
\begin{equation*}
\D(P_s):=\{u\in L^2E_1\,|\,\exists\{u_j\}_{j\in\N}\subset \Gamma_0(M,E_1),\ \exists v\in L^2E_2 \text{ s.t.}\ u_j\to u,\ Pu_j\to v\},
\end{equation*}
and $P_s u:=v$.

The \emph{weak extension} $P_w$ (also called the \emph{maximal closed extension} $P_{max}$) is defined as the largest extension of ${P_0}$ which acts distributionally, \textit{i.e.},
\begin{equation*}
\D(P_w):=\{u\in L^2E_1\,|\,\exists v\in L^2E_2,\ \text{s.t.} \ \la v,w\ra_2=\la u,P^*w\ra_1, \ \forall w\in \Gamma_0(M,E_2)\},
\end{equation*}
and $P_w u:=v$. Here $\la\cdot, \cdot \ra_i$ denotes the inner product defined on $L^2 E_i$. Note that this definition is equivalent to saying $P_w$ is the Hilbert adjoint of $P^*$ restricted to smooth forms with compact support, \textit{i.e.}, $((P^*)_0)^t=P_w$. Moreover $P_s\subseteq P_w$, and every closed extension $P'$ of $P_0$ which acts distributionally (namely, such that $(P')^t$ is an extension of $(P^*)_0$) is contained between the minimal and the maximal closed extensions. 

\begin{remark}\label{rmk strong-weak}
Since a closable densely defined operator and its closure have the same adjoint by Lemma \ref{lem dem adjoint ker im}, it follows that $((P^*)_s)^t=P_w$, thus implying
\begin{equation*}\label{eq strong-weak}
(P^*)_s=(P_w)^t,\ \ \ (P^*)_w=(P_s)^t.
\end{equation*}
Again from Lemma \ref{lem dem adjoint ker im}, we immediately get
\begin{equation*}\label{eq perp ker im strong weak}
(\ker P_w)^{\perp}=\c{\im P^*_s}\ \ \ \ \ (\ker P_s)^{\perp}=\c{\im P^*_w}.
\end{equation*}
Moreover, note that that if $E_1=E_2$ and $P$ is formally self-adjoint, then $P_0$ is essentially self-adjoint if and only if $P_s=P_w$.
\end{remark}

Set $E:=E_1=E_2$ and let ${P},{Q}:\Gamma(M,E)\to\Gamma(M,E)$ be differential operators and denote by $P_0,Q_0:\Gamma_0(M,E)\to\Gamma_0(M,E)$ their restrictions to compactly supported sections. It is straightforward to adapt the following theory in the case where ${P}$ and ${Q}$ are defined from sections of $(E_1,h_1)$ to sections of another Hermitian vector bundle $(E_2,h_2)$.

We now provide a number of technical but well-known lemmas which will be needed in subsequent sections.

\begin{lemma}\label{lemma closure strong}$
\c{\im P_0}=\c{\im P_s}.
$
\end{lemma}
\begin{proof}
The inclusion $\subseteq$ follows from $\im P_0\subseteq{\im P_s}$ passing to the closures. The other inclusion $\supseteq$ follows from $\c{\im P_0}\supseteq{\im P_s}$ passing again to the closures.
\end{proof}

\begin{lemma}\label{lemma image composition strong}
$
\c{\im (QP)_s}\subseteq \c{\im Q_s}.
$
\end{lemma}
\begin{proof}
By Lemma \ref{lemma closure strong}, $\c{\im (QP)_s}=\c{\im (QP)_0}$ and $\c{\im Q_s}=\c{\im Q_0}$, therefore it is sufficient to note that $\im (QP)_0\subseteq\im Q_0$ and then pass to the closures.
\end{proof}

\begin{lemma}\label{lemma kernel composition weak}
$
\ker P_w\subseteq \ker(QP)_w.
$
\end{lemma}
\begin{proof}
If $\alpha\in\ker P_w$, then for every $\gamma$ smooth section with compact support
\[
\la \alpha, (QP)^*\gamma\ra=\la \alpha,P^*Q^*\gamma\ra=\la P_w\alpha,Q^*\gamma\ra=0,
\]
therefore $\alpha\in\ker(QP)_w$.
\end{proof}

Let us now introduce the order relation $s\le w$ between strong and weak extensions, as well as the notation $a'=w$ if $a=s$ and $a'=s$ if $a=w$.

\begin{lemma}\label{lemma pq0 subset}
If $QP= 0$, then
$\c{\im P_a}\subseteq\ker Q_b$ for any $a,b\in\{s,w\}$ with $a\le b$.
\end{lemma}
\begin{proof}
It is enough to prove ${\im P_a}\subseteq\ker Q_a$ and pass to the closures.
If $a=s$ this property follows from the definition, while if $a=w$ it follows from the definition and the relation $(QP)^*=P^*Q^*$.
\end{proof}

\begin{lemma}\label{lemma perp pq0}
If $QP =0$, then $\im Q^*_{c}\perp\im P_a$ for any $a,c\in\{s,w\}$ with $\min(a,c)=s$.
\end{lemma}
\begin{proof}
The result is equivalent to $\im Q^*_{b'}\perp\im P_a$ for $a,b\in\{s,w\}$ with $a\le b$, which follows from Theorem \ref{thm decomp hilb dem} and Lemma \ref{lemma pq0 subset}.
\end{proof}

\begin{lemma}\label{lemma smooth strong weak}
For $a\in\{s,w\}$, we have ${P} = P_a$ when acting on $\Gamma(M,E)\cap\D(P_a)$.
\end{lemma}
\begin{proof}
Since $P_s\subseteq P_w$, it is enough to take $\alpha\in \Gamma(M,E)\cap\D(P_w)$ and note that for every $\gamma\in\Gamma_0(M,E)$
\[
\la {P}\alpha-P_w\alpha,\gamma\ra=\la \alpha, P^*\gamma-P^*\gamma\ra=0,
\]
therefore by the density of $\Gamma_0(M,E)$ in $L^2E$, the previous equation holds for every $\gamma$ in $L^2$, implying ${P}\alpha=P_w\alpha$ (for every $x\in M$ it is enough to choose, \textit{e.g.}, $\gamma=({P}\alpha-P_w\alpha)1_{K_x}$, where $1_{K_x}$ is the function equal to $1$ on the compact $K_x\ni x$ and equal to $0$ outside).
\end{proof}

\begin{lemma}\label{lemma kernel weak smooth}
It holds that
$$\D(P_w)\cap\Gamma(M,E)=\{\alpha\in L^2E\cap \Gamma(M,E)\,|\,{P}\alpha\in L^2E\}$$ and
$$\ker {P}_w\cap \Gamma(M,E)=\{\alpha\in L^2E\cap \Gamma(M,E)\,|\,{P}\alpha=0\}.$$
\end{lemma}
\begin{proof}
If $\alpha\in L^2E\cap \Gamma(M,E)$, then for every $\gamma\in\Gamma_0(M,E)$ smooth section with compact support
\[
\la {P}\alpha,\gamma\ra=\la\alpha, P^*\gamma\ra.
\]
Therefore, arguing as in Lemma \ref{lemma smooth strong weak}, $\alpha\in \D(P_w)$ iff ${P}\alpha\in L^2E$, and analogously
 $\alpha\in\ker {P}_w$ iff ${P}\alpha=0$.
\end{proof}

\begin{remark}
The definitions of strong and weak extensions and the previous technical lemmas can be given in a more general setting. 

For any Hilbert space $\H$ that contains $\Gamma_0(M,E)$ as a dense subspace, it it possible to define the strong and weak extensions of $P_0:\Gamma_0(M,E)\to\Gamma_0(M,E)$ so long as it admits a formal adjoint with respect to the inner product of $\H$; with the same assumptions one can go on to prove Lemmas 5.2-5.6. As a consequence, the $L^2$ Dolbeault and the $L^2$ Aeppli-Bott-Chern Hilbert complexes, which will be defined in Sections \ref{sec hilb complex dolbeault} and \ref{sec hilb complex bc a}, may also be constructed for Hilbert spaces other than $L^2$. To prove also Lemmas 5.7 and 5.8 in general one would need additional assumptions on $\H$, so that its inner product looks sufficiently like the $L^2$ inner product. 

In \cite{AV} Andreotti and Vesentini introduced a Hilbert space $W^{p,q}$ defined as $\D((\delbar+\delbar^*)_s)\cap L^2\Lambda^{p,q}$ on a Hermitian manifold. If the Hermitian metric is K\"ahler, in \cite[Section 5]{P} the second author proved that for a class of differential operators (including $d,d^*,\del,\del^*,\delbar,\delbar^*$ and all the Laplacians introduced in Section \ref{sec complex manifold}) the $W^{p,q}$ formal adjoint coincide with the $L^2$ formal adjoint. Therefore, K\"ahler manifolds with the Hilbert space $W^{p,q}$ provide an example of a more general setting where strong and weak extensions can be defined and Lemmas 5.2-5.8 hold.
\end{remark}

\section{The Dolbeault Hilbert complex}
\label{sec hilb complex dolbeault}
Let $(M,g)$ be a Hermitian manifold of complex dimension $n$.
For any fixed $(p,q)$, let us consider the following \emph{$L^2$ Dolbeault Hilbert complex}
\begin{equation}\label{eq dolbeault hilbert complex}
\dots{\longrightarrow}L^2\Lambda^{p,q-1}\overset{\delbar_a}{\longrightarrow}L^2\Lambda^{p,q}\overset{\delbar_b}{\longrightarrow}L^2\Lambda^{p,q+1}{\longrightarrow}\dots
\end{equation}
where $a,b\in\{s,w\}$ with $a\le b$ and so, by Lemma \ref{lemma pq0 subset}, we have  $\overline{\im \delbar_a} \subseteq \ker \delbar_b$ and consequently the cohomology spaces below are well-defined.

We denote by
\[
L^2H^{p,q}_{\delbar,ab}:=\frac{L^2\Lambda^{p,q}\cap\ker\delbar_b}{L^2\Lambda^{p,q}\cap\im\delbar_a}
\]
the associated \emph{unreduced $L^2$ Dolbeault cohomology},
and by
\[
L^2\bar{H}^{p,q}_{\delbar,ab}:=\frac{L^2\Lambda^{p,q}\cap\ker\delbar_b}{L^2\Lambda^{p,q}\cap\c{\im\delbar_a}},
\]
the associated \emph{reduced $L^2$ Dolbeault cohomology}.

As described in \eqref{Hilbert complex laplacian}, the Hilbert complex \eqref{eq dolbeault hilbert complex} has an associated Laplacian, which is a positive self-adjoint operator given by 
\[
\Delta_{\delbar,ab}:=\delbar_a\delbar_a^t+\delbar_b^t\delbar_b,
\]
defined on $L^2\Lambda^{p,q}$ with domain given by Theorems \ref{thm composition with adjoint} and \ref{thm sum of self ajoint}. In fact, the operator $\Delta_{\delbar,ab}$ is well-defined even if $a>b$.
 
Recalling Remark \ref{rmk strong-weak} and the notation $a'=w$ if $a=s$ and $a'=s$ if $a=w$, note that $\Delta_{\delbar,ab}$  can be rewritten as
\[
\Delta_{\delbar,ab}=\delbar_a\delbar^*_{a'}+\delbar^*_{b'}\delbar_b.
\]
With this formulation of $\Delta_{\delbar,ab}$, it is easy to see that it is an extension of the Dolbeault Laplacian $(\Delta_{\delbar})_0$ acting on the space $A_0^{p,q}$ of smooth compactly supported forms.

By Lemma \ref{lemma ker lapl} it holds that in $L^2 \Lambda^{p,q}$ we have
\begin{equation}\label{eq kernel dolbeault laplacian ab}
L^2\H^{p,q}_{\delbar,ab}:=\ker\Delta_{\delbar,ab}=\ker\delbar_b\cap\ker\delbar^*_a,
\end{equation}
where $L^2\H^{p,q}_{\delbar,ab}$ is the space of \emph{$L^2$-Dolbeault harmonic forms}.

Moreover, by Theorem \ref{thm decomp hilb dem}, we obtain the following Dolbeault orthogonal decomposition of the Hilbert space $L^2\Lambda^{p,q}$
\begin{align*}
L^2\Lambda^{p,q}&=L^2\H^{p,q}_{\delbar,ab}\oplus\c{\im\delbar_a}\oplus\c{\im\delbar^*_{b'}},\\
\ker \delbar_b&=L^2\H^{p,q}_{\delbar,ab}\oplus\c{\im\delbar_a}.
\end{align*}
From this last decomposition we immediately deduce the isomorphism, induced by the identity, between the space of $L^2$ Dolbeault harmonic forms and $L^2$ reduced cohomology
\[
L^2\H^{p,q}_{\delbar,ab}\simeq L^2\bar{H}^{p,q}_{\delbar,ab}.
\]
Note that by elliptic regularity, \textit{i.e.}, Theorem \ref{thm ell-reg}, we get $L^2\H^{p,q}_{\delbar,ab}\subseteq A^{p,q}$.

\begin{remark}
This theory of the $L^2$ Dolbeault Hilbert complex holds similarly for the $L^2$ Hilbert complexes originated by $\del^2=0$ and $d^2=0$. In particular, for $a,b\in\{s,w\}$, we can define self-adjoint operators $\Delta_{\del,ab}$, $\Delta_{d,ab}$, $L^2$ cohomology spaces $L^2H^{p,q}_{\del,ab}$, $L^2H^{k}_{d,ab}$, $L^2\bar{H}^{p,q}_{\del,ab}$, $L^2\bar{H}^{k}_{d,ab}$ and spaces of $L^2$ harmonic forms $L^2\H^{p,q}_{\del,ab}$, $L^2\H^{k}_{d,ab}$ with analogous properties.
\end{remark}

\section{The Aeppli-Bott-Chern Hilbert complex}
\label{sec hilb complex bc a}
Throughout this section, $(M,g)$ will denote a Hermitian manifold of complex dimension $n$.

For any fixed bidegree $(p,q)$, we define the \emph{$L^2$ Aeppli-Bott-Chern Hilbert complex}, or \emph{$L^2$ ABC complex} for short, to be
\begin{equation}\label{eq bc hilbert complex}
\begin{tikzcd}
\dots\arrow[d]\\
L^2\Lambda^{p-1,q-2}\oplus L^2\Lambda^{p-2,q-1}\arrow[d,"(\delbar \oplus \del)_a"]\\
L^2\Lambda^{p-1,q-1}\arrow[d,"\del\delbar_b"]\\
L^2\Lambda^{p,q}\arrow[d,"(\del+\delbar)_c"]\\
L^2\Lambda^{p+1,q}\oplus L^2\Lambda^{p,q+1}\arrow[d]\\
\dots
\end{tikzcd}
\end{equation}
where $a,b,c\in\{s,w\}$ with $a\le b\le c$ denote either strong or weak extensions. Note that we choose to omit the parentheses when writing $\del \delbar_b := (\del \delbar)_b$ for simplicity of notation. The complete definition of the $L^2$ ABC Hilbert complex is given in Section \ref{sec questions}.

By Lemma \ref{lemma pq0 subset} we have 
\[
\c{\im(\delbar \oplus \del)_a}\subseteq\ker \del\delbar_b,\ \ \ \c{\im\del\delbar_b}\subseteq\ker (\del+\delbar)_c,
\]
and therefore the associated \emph{unreduced $L^2$ Bott-Chern} and \emph{Aeppli cohomology} spaces, defined by
\[
L^2H^{p,q}_{BC,bc}:=\frac{L^2\Lambda^{p,q}\cap\ker (\del+\delbar)_c}{L^2\Lambda^{p,q}\cap\im\del\delbar_b}
\]
and
\[
L^2H^{p-1,q-1}_{A,ab}:=\frac{L^2\Lambda^{p-1,q-1}\cap\ker\del\delbar_b}{L^2\Lambda^{p-1,q-1}\cap\im(\delbar \oplus \del)_a},
\]
are well-defined. 
Similarly, we can define the \emph{reduced $L^2$ Bott-Chern} and \emph{Aeppli cohomology} spaces
\[
L^2\bar{H}^{p,q}_{BC,bc}:=\frac{L^2\Lambda^{p,q}\cap\ker (\del+\delbar)_c}{L^2\Lambda^{p,q}\cap\c{\im\del\delbar_b}}
\]
and
\[
L^2\bar{H}^{p-1,q-1}_{A,ab}:=\frac{L^2\Lambda^{p-1,q-1}\cap\ker\del\delbar_b}{L^2\Lambda^{p-1,q-1}\cap\c{\im(\delbar \oplus \del)_a}}.
\]

The positive self-adjoint operators associated to the the Hilbert complex \eqref{eq bc hilbert complex} are
\[
\Delta_{BC,bc}:=\del\delbar_b(\del\delbar_b)^t+(\del+\delbar)_c^t(\del+\delbar)_c:L^2\Lambda^{p,q}\to L^2\Lambda^{p,q}
\]
and
\[
\Delta_{A,ab}:=(\del\delbar_b)^t\del\delbar_b+(\delbar \oplus \del)_a(\delbar \oplus \del)_a^t:L^2\Lambda^{p-1,q-1}\to L^2\Lambda^{p-1,q-1},
\]
with domains given by Theorems \ref{thm composition with adjoint} and \ref{thm sum of self ajoint}. By analogy with the elliptic complex \eqref{eq bc complex} we also define the following positive self-adjoint operators as $L^2$ versions of the operators \eqref{eq defin ell lapl}.
\[
\square_{BC,bc}=\del\delbar_b(\del\delbar_b)^t+((\del+\delbar)_c^t(\del+\delbar)_c)^2,
\]
acting on $L^2 \Lambda^{p,q}$, and
\[
\square_{A,ab}=(\del\delbar_b)^t\del\delbar_b+((\delbar \oplus \del)_a(\delbar \oplus \del)_a^t)^2,
\]
acting on $L^2\Lambda^{p-1,q-1}$.
 Their domains are given by Theorems \ref{thm composition with adjoint} and \ref{thm sum of self ajoint}.

Note that without the condition $a\leq b \leq c$, the above operators are still well-defined, however it may not be possible to define the corresponding cohomology spaces. 

By Remark \ref{rmk strong-weak} and recalling the notation $a'=w$ if $a=s$ and $a'=s$ if $a=w$, the operators in the $L^2$ ABC Hilbert complex have adjoints given by 
\begin{equation}\label{eq adjoints abc hilb complex}
(\delbar \oplus \del)_a^t = (\del^*+\delbar^*)_{a'}, \ \ \  \del \delbar_b^t = \delbar^* \del^*_{b'}, \ \ \   (\del+\delbar)_c^t=(\del^*\oplus\delbar^*)_{c'}.
\end{equation}
This allows us to rewrite the operators $\Delta_{BC,bc}$ and $\Delta_{A,ab}$ as
\[
\Delta_{BC,bc}=\del\delbar_b\delbar^*\del^*_{b'}+(\del^*\oplus\delbar^*)_{c'}(\del+\delbar)_c,
\]
\[
\Delta_{A,ab}=\delbar^*\del^*_{b'}\del\delbar_b+(\delbar\oplus\del)_a(\del^*+\delbar^*)_{a'},
\]
and the operators $\square_{BC,bc}$ and $\square_{A,ab}$ as
\[
\square_{BC,bc}=\del\delbar_b\delbar^*\del^*_{b'}+((\del^*\oplus\delbar^*)_{c'}(\del+\delbar)_c)^2,
\]
\[
\square_{A,ab}=\delbar^*\del^*_{b'}\del\delbar_b+((\delbar\oplus\del)_a(\del^*+\delbar^*)_{a'})^2.
\]
This makes it clear that the above operators are just extensions of the non-elliptic operators $(\Delta_{BC})_0$, $(\Delta_{A})_0$ and the elliptic operators $(\square_{BC})_0$, $(\square_{A})_0$, acting on the space $A^{\bullet,\bullet}_0$ of smooth forms with compact support. 

\begin{remark}
The dual Hilbert complex of \eqref{eq bc hilbert complex} is given by
\begin{equation*}\label{eq bc dual hilbert complex}
\begin{tikzcd}
L^2\Lambda^{p-1,q-2}\oplus L^2\Lambda^{p-2,q-1}\\
L^2\Lambda^{p-1,q-1}\arrow[u,"(\del^* + \delbar^*)_{a'}"]\\
L^2\Lambda^{p,q}\arrow[u,"\delbar^*\del^*_{b'}"]\\
L^2\Lambda^{p+1,q}\oplus L^2\Lambda^{p,q+1}\arrow[u,"(\del^*\oplus\delbar^*)_{c'}"]
\end{tikzcd}
\end{equation*}
with $a,b,c\in\{s,w\}$ chosen such that $a\le b\le c$. This will be called the \emph{dual $L^2$ ABC Hilbert complex}.
\end{remark}

By Lemma \ref{lemma ker lapl}, we see that the kernels of the elliptic and non-elliptic operators coincide, in particular we have
\[
\ker\Delta_{BC,bc}=\ker\square_{BC,bc}=\ker\delbar^*\del^*_{b'}\cap\ker (\del+\delbar)_c,
\]
on $L^2 \Lambda^{p,q}$, and
\[
\ker\Delta_{A,ab}=\ker\square_{A,ab}=\ker (\del^*+\delbar^*)_{a'}\cap\ker\del\delbar_b.
\]
on $L^2\Lambda^{p-1,q-1}$.
These kernels will be denoted by
\begin{align*}
L^2\H^{p,q}_{BC,bc}&:=\ker\Delta_{BC,bc}\cap L^2\Lambda^{p,q},\\
L^2\H^{p-1,q-1}_{A,ab}&:=\ker\Delta_{A,ab}\cap L^2\Lambda^{p-1,q-1},
\end{align*}
and we will call them the spaces of \emph{$L^2$ Bott-Chern  harmonic forms} and \emph{$L^2$ Aeppli harmonic forms}, respectively. Since $\square_{BC}$ and $\square_{A}$ are elliptic and $\ker\Delta_{BC,bc}\subseteq \ker(\square_{BC})_w$, $\ker\Delta_{A,ab}\subseteq \ker(\square_{A})_w$, it follows from Theorem \ref{thm ell-reg} (elliptic regularity) that these spaces of $L^2$ harmonic forms are smooth, namely 
\[
L^2\H^{p,q}_{BC,bc}\subseteq A^{p,q},\ \ \ L^2\H^{p,q}_{A,ab}\subseteq A^{p,q}.
\]

\begin{remark}
    As the Hodge $*$ operator is an isometry with respect to the $L^2$ inner product, the following duality between spaces of harmonic forms is obtained as a consequence of \eqref{eq duality bc a}.
\[
L^2\H^{p,q}_{BC,bc}\simeq L^2\H^{n-q,n-p}_{A,c'b'}.
\]
\end{remark}

\begin{remark}
The $L^2$ ABC complex \eqref{eq bc hilbert complex} is composed by two separate Hilbert complexes: the \emph{$L^2$ Aeppli Hilbert complex} is given by the first two differentials $(\delbar\oplus\del)_a$ and $\del\delbar_b$, with $a,b\in\{s,w\}$ and $a\le b$; while the \emph{$L^2$ Bott-Chern Hilbert complex} is given by the last two differentials $\del\delbar_b$ and $(\del+\delbar)_c$, with $b,c\in\{s,w\}$ and $b\le c$.

In the following we might consider these two Hilbert complexes separately.
In this way, without any loss of generality, we can increment by $(1,1)$ the bidegree for the $L^2$ Aeppli Hilbert complex in \eqref{eq bc hilbert complex}. As a consequence, we will uniform the subscripts denoting strong or weak extensions in both Hilbert complexes to just $a,b$ with $a\le b$.
\end{remark}

We introduce the notation $L^2A^{p,q}:=L^2\Lambda^{p,q}\cap A^{p,q}$ to denote the space of $L^2$ $(p,q)$-forms which are smooth. The spaces $L^2A^{k}$ and $L^2A^{k}_\C$ can be defined similarly.
By applying Theorem \ref{thm decomp hilb dem} to the $L^2$ ABC Hilbert complex, we obtain the following orthogonal decompositions of the Hilbert space $L^2\Lambda^{p,q}$. 
\begin{theorem}[$L^2$ Bott-Chern and Aeppli decompositions]\label{thm l2 bc decomp}
The Hilbert space $L^2\Lambda^{p,q}$ decomposes as
\begin{align*}
L^2\Lambda^{p,q}&=L^2\H^{p,q}_{BC,ab}\oplus\c{\im\del\delbar_a}\oplus\c{\im(\del^*\oplus\delbar^*)_{b'}},\\
L^2\Lambda^{p,q}&=L^2\H^{p,q}_{A,ab}\oplus\c{\im(\delbar\oplus \del)_a}\oplus\c{\im\delbar^*\del^*_{b'}},
\end{align*}
where $a,b\in\{s,w\}$ such that $a\le b$. We also have 
\begin{align*}
\ker (\del+\delbar)_b&=L^2\H^{p,q}_{BC,ab}\oplus\c{\im\del\delbar_{a}},\\
\ker\del\delbar_b&=L^2\H^{p,q}_{A,ab}\oplus\c{\im(\delbar\oplus\del)_{a}}.
\end{align*}
Moreover a smooth form $\alpha \in L^2 A^{p,q}$ has smooth components with respect to the above decompositions.
\end{theorem}
\begin{proof}
The decompositions follow from Theorem \ref{thm decomp hilb dem}, therefore it only remains to prove the regularity.
Given $\alpha\in L^2\Lambda^{p,q}$ we can write
\[
\alpha=h+\beta+\eta
\]
by the first decomposition, where $h\in\ker\square_{BC,ab}$, $\beta\in\c{\im\del\delbar_a}$ and $\eta\in\c{\im(\del^*\oplus\delbar^*)_{b'}}$.
Lemma \ref{lemma pq0 subset} tells us that $\beta\in\ker\del_w\cap\ker\delbar_w$ and $\eta\in\ker\delbar^*\del^*_w$, which implies that, when $\alpha$ is smooth, $\beta$ is a weak solution of 
\[
\square_{BC}\beta=\del\delbar\delbar^*\del^*\alpha
\]
 and $\eta$ is a weak solution of 
\[
\square_{BC}\eta=(\del^*\del+\delbar^*\delbar)^2\alpha.
\]
By Theorem \ref{thm ell-reg}, it follows that $\beta$ and $\eta$ are also smooth, proving the result. The Aeppli case is analogous.
\end{proof}
\begin{corollary}
    There exist isomorphisms, induced by the identity, between the spaces of $L^2$ harmonic forms and reduced $L^2$ cohomology for $a,b\in\{s,w\}$ such that $a\le b$
    \[
L^2\H^{p,q}_{BC,ab}\simeq L^2\bar{H}^{p,q}_{BC,ab}\ \ \ \ \ \ \ \ \ L^2\H^{p,q}_{A,ab}\simeq L^2\bar{H}^{p,q}_{A,ab}.
\]
\end{corollary}

By \cite[Lemma 3.8]{BL} if $(\Delta_{BC})_0$ or $(\Delta_{A})_0$ are essentially self-adjoint, then $\del\delbar_s=\del\delbar_w$ and $\delbar^*\del^*_s=\delbar^*\del^*_w$. However, since $(\Delta_{BC})_0$ and $(\Delta_{A})_0$ are not elliptic, it can be difficult to find actual cases where they are essentially self-adjoint.
Instead, we prove that the same result holds when $(\square_{BC})_0$ or $(\square_{A})_0$ are essentially self-adjoint; this will become important in Section \ref{sec coverings} (see the discussion after Proposition \ref{pullback self adj}). 

\begin{theorem}\label{thm bc essentially deldelbar}
If $(\square_{BC})_0$ or $(\square_{A})_0$ is essentially self-adjoint, then
\[
\del\delbar_s=\del\delbar_w,\ \ \ \delbar^*\del^*_s=\delbar^*\del^*_w.
\]
\end{theorem}
\begin{proof}
Given a closed linear operator $P:\mathcal{H}_1 \rightarrow \mathcal{H}_2$, a linear subspace $\mathcal{E}\subseteq \D(P)$ is called a \emph{core} if $\mathcal{E}$ is dense in $\D(P)$ with respect to the graph norm $\lv x \rv_{Gr(P)} := \lv x \rv_{\mathcal{H}_1} + \lv Px \rv_{\mathcal{H}_2}$. 

By \cite[p. 98]{BL} we know that
\[
\mathcal{E}_{ab}:=\bigcap_{k\in\N}\D(\Delta_{A,ab}^k)
\]
is a core for $\del\delbar_b$ with $a,b\in\{s,w\}$ and $a\le b$. 
But by definition $\mathcal{E}_{ab}\subseteq\D(\Delta_{A,ab}^2)$, and by \eqref{eq domain square lapl} we have $\D(\Delta_{A,ab}^2)\subseteq\D(\square_{A,ab})$, therefore
$\D(\square_{A,ab})$ is also a core for $\del\delbar_b$. In particular $\D(\square_{A,ss})$ is a core for $\del\delbar_s$ and $\D(\square_{A,sw})$ is a core for $\del\delbar_w$.
Now note that since $\del\delbar_s\subseteq\del\delbar_w$, we know that $\del\delbar_s$ and $\del\delbar_w$ coincide on $\D(\square_{A,ss})$, and if $(\square_{A})_0$ is essentially self-adjoint then $\D(\square_{A,ss})=\D(\square_{A,sw})$.
By the definition of core this shows $\del\delbar_s=\del\delbar_w$, and Remark \ref{rmk strong-weak} ends the proof in the case when $(\square_{A})_0$ is essentially self-adjoint.

If $(\square_{BC})_0$ is essentially self-adjoint, considering the dual $L^2$ ABC Hilbert complex \eqref{eq bc dual hilbert complex}, a similar argument shows that $\delbar^*\del^*_s=\delbar^*\del^*_w$, and again Remark \ref{rmk strong-weak} ends the proof.
\end{proof}

\begin{remark}
Actually, \cite[Lemma 3.8]{BL} yields that if $(\Delta_{BC})_0$ is essentially self-adjoint, then $\del\delbar_s=\del\delbar_w$ and $(\del+\delbar)_s=(\del+\delbar)_w$. Similarly, if $(\Delta_{A})_0$ is essentially self-adjoint, then $\del\delbar_s=\del\delbar_w$ and $(\delbar\oplus\del)_s=(\delbar\oplus\del)_w$. With the same proof of Theorem \ref{thm bc essentially deldelbar} we can prove that if $(\square_{BC})_0$ is essentially self-adjoint, then $(\del+\delbar)_s=(\del+\delbar)_w$, and if $(\square_{A})_0$ is essentially self-adjoint, then $(\delbar\oplus\del)_s=(\delbar\oplus\del)_w$. However, since $\del+\delbar$ and $\delbar\oplus\del$ are first order differential operators, in actual examples the two properties $(\del+\delbar)_s=(\del+\delbar)_w$ and $(\delbar\oplus\del)_s=(\delbar\oplus\del)_w$ hold when the metric is complete (see, \textit{e.g.}, \cite[Theorem 1.3]{A}), which usually is a more frequent assumption to get, compared with essential self-adjointness.
\end{remark}

By a similar argument, if $(\square_{BC})_0$ and/or $(\square_{A})_0$ are essentially self-adjoint, we are able to characterise the kernel of their unique self-adjoint extension.
First note that, as a consequence of Lemma \ref{lemma kernel weak smooth}, we have 
\begin{align*}
\ker\square_{BC,sw}&=\{\alpha\in L^2A^{\bullet,\bullet}\,|\,\del\alpha=\delbar\alpha=\delbar^*\del^*\alpha=0\},\\
\ker\square_{A,sw}&=\{\alpha\in L^2A^{\bullet,\bullet}\,|\,\del^*\alpha=\delbar^*\alpha=\del\delbar\alpha=0\}.
\end{align*}
Furthermore, again by Lemma \ref{lemma kernel weak smooth}, we have 
\begin{align*}
\ker (\square_{BC})_w&=\{\alpha\in L^2A^{\bullet,\bullet}\,|\,\square_{BC}\alpha=0\},\\
\ker (\square_{A})_w&=\{\alpha\in L^2A^{\bullet,\bullet}\,|\,\square_{A}\alpha=0\}.
\end{align*}
Therefore we obtain the following result.

\begin{proposition}\label{prop bc essentially harmonic}
 If $(\square_{BC})_0$ is essentially self-adjoint, then for any smooth form $\alpha\in L^2 A^{\bullet,\bullet}$ it holds that
\[
\square_{BC}\alpha=0\ \ \ \iff\ \ \ \del\alpha=\delbar\alpha=\delbar^*\del^*\alpha=0.
\]
If $(\square_{A})_0$ is essentially self-adjoint, then for any smooth form $\alpha\in L^2 A^{\bullet,\bullet}$ it holds that
\[
\square_{A}\alpha=0\ \ \ \iff\ \ \ \del^*\alpha=\delbar^*\alpha=\del\delbar\alpha=0.
\]
\end{proposition}

As a direct consequence of \cite[Theorem 3.5]{BL}, the unreduced $L^2$ Bott-Chern and Aeppli cohomologies can be computed using the subcomplex obtained by intersecting \eqref{eq bc hilbert complex} with the space of smooth forms $A^{\bullet,\bullet}$. Note that this subcomplex is not itself a Hilbert complex.

\begin{theorem}[{\cite[Theorem 3.5]{BL}}]\label{thm isom cohm smooth}
We have the following isomorphisms for $a,b\in\{s,w\}$ and $a\le b$
\[
L^2H^{p,q}_{BC,ab}\simeq \frac{\ker(\del+\delbar)_b\cap A^{p,q}}{\del\delbar_a(A^{p-1,q-1}\cap\D(\del\delbar_a))},
\]
\[
L^2H^{p,q}_{A,ab}\simeq \frac{\ker \del\delbar_b\cap A^{p,q}}{(\delbar\oplus\del)_a((A^{p,q-1}\oplus A^{p-1,q})\cap\D((\delbar\oplus\del)_a))}.
\]
\end{theorem}

We note that, in view of Lemma \ref{lemma kernel weak smooth}, when $a=b=c=w$ the statement simplifies to
\[
L^2H^{p,q}_{BC,ww}\simeq \frac{\ker\del\cap\ker\delbar\cap L^2A^{p,q}}{\del\delbar(L^2A^{p-1,q-1}) \cap L^2A^{p,q}},
\]
\[
L^2H^{p,q}_{A,ww}\simeq \frac{\ker \del\delbar\cap L^2A^{p,q}}{(\delbar\oplus\del)(L^2A^{p,q-1}\oplus L^2A^{p-1,q})\cap L^2 A^{p,q}}.
\]

As a consequence of the above results, on a compact Hermitian manifold the spaces of Bott-Chern cohomology,  Bott-Chern harmonic forms,  $L^2$ reduced Bott-Chern cohomology,  $L^2$ unreduced Bott-Chern cohomology, and  $L^2$ Bott-Chern harmonic forms are all isomorphic.

Indeed, in this setting $(\square_{BC})_0$ is essentially self-adjoint \cite[Proposition 4.1]{Ati64} and so $\del\delbar_s=\del\delbar_w$ by Theorem \ref{thm bc essentially deldelbar}. Furthermore $(\del+\delbar)_s=(\del+\delbar)_w$ and $(\delbar\oplus\del)_s=(\delbar\oplus\del)_w$ since the metric $g$ is complete (see, \textit{e.g.}, \cite[Theorem 1.3]{A}). Therefore
\[
L^2H^{p,q}_{BC,bc}\simeq L^2H^{p,q}_{BC,ww}.
\]
Moreover, by Theorem \ref{thm isom cohm smooth} and Lemma \ref{lemma kernel weak smooth}, we have
\[
L^2H^{p,q}_{BC,ww}\simeq H^{p,q}_{BC},
\]
and since $M$ is compact $H^{p,q}_{BC}\simeq\H^{p,q}_{BC}$ has finite dimension. This implies that $\im\del\delbar_b$ is closed (\textit{cf.} \cite[Theorem 2.4]{BL}), thus
\[
L^2H^{p,q}_{BC,bc}\simeq L^2\bar{H}^{p,q}_{BC,bc}\simeq L^2\H^{p,q}_{BC,bc}.
\]

Similar isomorphisms also exist for Aeppli cohomology.

\subsection{A remark on the reduced \texorpdfstring{$L^2$}{L2} Aeppli and Bott-Chern cohomologies}\label{subsec remark aeppli}\,\\
Recall the smooth Aeppli and Bott-Chern cohomologies, considered in Section \ref{sec complex manifold}
$$ 
H^{p,q}_{BC}:=\frac{\ker \del +\delbar}{\im\del\delbar},\quad \quad  H^{p,q}_A:=\frac{\ker \del\delbar}{\im\delbar\oplus\del}.$$
When the bidegree is fixed we can write $\ker \del+\delbar$ and $\im \del \oplus \delbar$ equivalently as $\ker \del \cap \ker \delbar$ and $\im \del + \im \delbar$.
However, things are not so simple for the reduced $L^2$ Aeppli and Bott-Chern cohomologies
\[
L^2\bar{H}^{p,q}_{BC,ab}:=\frac{\ker (\del+\delbar)_b}{\c{\im\del\delbar_a}},
\quad \quad
L^2\bar{H}^{p,q}_{A,ab}:=\frac{\ker\del\delbar_b}{\c{\im(\delbar \oplus \del)_a}},
\]
with $a,b \in \{s,w\}$, $a \leq b$. There is no guarantee that we have $(\del + \delbar)_b = \del_b + \delbar_b$ and as a consequence we cannot assume that $\ker(\del+\delbar)_b$ and $\c{\im (\delbar\oplus \del)_a}$ can be used interchangeably with $\ker \del_b \cap \ker \delbar_b$ and $\c{\im \delbar_a + \im \del_a}$; instead we have the following results.
\begin{lemma}\label{lemma intersection kernel weak strong}
Let the operators $\del_a+\delbar_a$ and $\del^*_a+\delbar^*_a$ be defined with domains $\D(\del_a)\cap\D(\delbar_a) \subseteq L^2\Lambda^{p,q}$ and $\D(\del^*_a)\cap\D(\delbar^*_a) \subseteq L^2\Lambda^{p,q}$ respectively, for some choice of strong or weak extensions $a \in \{s,w\}$.
Then for the weak extension on $L^2\Lambda^{p,q}$ we have
\[
(\del+\delbar)_w=\del_w+\delbar_w,\ \ \ (\del^*+\delbar^*)_w=\del^*_w+\delbar^*_w
\] 
however for the strong extension we have only
\[
(\del+\delbar)_s\subseteq\del_s+\delbar_s,\ \ \ (\del^*+\delbar^*)_s\subseteq\del^*_s+\delbar^*_s.
\]
In particular, on $L^2\Lambda^{p,q}$ we have
\begin{gather*}
\ker(\del+\delbar)_w=\ker\del_w\cap\ker\delbar_w, \quad \quad \ker(\del^*+\delbar^*)_w=\ker\del^*_w\cap\ker\delbar^*_w,
\end{gather*}
for the weak extension but only
\begin{gather*}
\ker(\del+\delbar)_s\subseteq\ker\del_s\cap\ker\delbar_s,\quad \quad 
\ker(\del^*+\delbar^*)_s\subseteq\ker\del^*_s\cap\ker\delbar^*_s,
\end{gather*}
for the strong extension.
\end{lemma}
\begin{proof}
The key observation is that $\del_w+\delbar_w,(\del+\delbar)_w,\del_s+\delbar_s,(\del+\delbar)_s$ are operators $L^2\Lambda^{p,q}\to L^2\Lambda^{p+1,q}\oplus L^2\Lambda^{p,q+1}$ and the spaces $L^2\Lambda^{p+1,q}$, $L^2\Lambda^{p,q+1}$ are orthogonal. Knowing this,
both inclusions $\del_w+\delbar_w\subseteq(\del+\delbar)_w$ and $\del_w+\delbar_w\supseteq(\del+\delbar)_w$ follow directly from the definitions, as well as $(\del+\delbar)_s\subseteq\del_s\cap\delbar_s$. The same works also for the operator $\del^*+\delbar^*$.
\end{proof}

An immediate consequence is the following.
\begin{corollary}\label{cor equality closure strong}
The following equalities hold in $L^2\Lambda^{p,q}$
$$\c{\im(\delbar \oplus \del)_s}=\c{\im\del_s}+\c{\im\delbar_s},\ \ \ \c{\im(\del^* \oplus \delbar^*)_s}=\c{\im\del^*_s}+\c{\im\delbar^*_s}.$$
\end{corollary}
\begin{proof}
We will just prove the first equality; the second is similar. Lemma \ref{lemma intersection kernel weak strong} tells us that we have 
\[
\ker(\del^*+\delbar^*)_w=\ker\del^*_w\cap\ker\delbar^*_w.
\]
in $L^2\Lambda^{p,q}$. By Remark \ref{rmk strong-weak}, we also have
\[
\ker(\del^*+\delbar^*)_w =\im(\delbar\oplus\del)_s^{\perp},
\]
and
\[
\ker\del^*_w\cap\ker\delbar^*_w=\im\del_s^\perp+\im\delbar_s^\perp=\left(\c{\im\del_s}+\c{\im\delbar_s}\right)^\perp.
\]
Therefore, the equality
\[
\c{\im(\delbar\oplus\del)_s}=\c{\im\del_s}+\c{\im\delbar_s}
\]
holds in $L^2\Lambda^{p,q}$.
\end{proof}

As a consequence, we have the following equivalent definitions of the reduced $L^2$ Aeppli and Bott-Chern cohomologies.
\begin{corollary}\label{cor defin aeppli equivalent}
When $a=s$ the reduced $L^2$ Aeppli cohomology can be written as
\[
L^2\bar{H}^{p,q}_{A,sb}=\frac{L^2\Lambda^{p,q}\cap\ker\del\delbar_b}{L^2\Lambda^{p,q}\cap\left(\c{\im\del_s}+\c{\im\delbar_s}\right)}.
\]
When $b = w$ the reduced $L^2$ Bott-Chern cohomology can be written as
\[
L^2\bar{H}^{p,q}_{BC,aw}=\frac{L^2\Lambda^{p,q}\cap\ker \del_w \cap \ker \delbar_w}{L^2\Lambda^{p,q}\cap\c{\im\del\delbar_b}}.
\]
\end{corollary}
This is especially useful when the Hermitian metric $g$ is complete and therefore the strong and the weak extensions of $\del,\delbar,\delbar\oplus\del$ coincide.

\begin{remark}\label{rmk clos image del+delbar}
Note that, by the same argument as in the proof of Lemma \ref{lemma closure strong}, we can obtain in $L^2\Lambda^{p,q}$
\begin{equation*}
\c{\im(\delbar\oplus\del)_s}=\c{\im\del_s+\im\delbar_s},\ \ \ \c{\im(\del^*\oplus\delbar^*)_s}=\c{\im\del^*_s+\im\delbar^*_s},
\end{equation*}
thus Corollary \ref{cor equality closure strong} implies
\[
\c{\im\del_s}+\c{\im\delbar_s}=\c{\im\del_s+\im\delbar_s},\ \ \ \c{\im\del^*_s}+\c{\im\delbar^*_s}=\c{\im\del^*_s+\im\delbar^*_s}.
\]
\end{remark}

\subsection{A diagram of maps between the reduced \texorpdfstring{$L^2$}{L2} cohomology spaces}\,\\
On a complex manifold we have the diagram \eqref{diagram maps cohomology} of maps between complex cohomologies. Given a Hermitian metric, we have a similar diagram involving $L^2$ reduced cohomology spaces.

\begin{proposition}\label{prop diagram maps cohomology hilbert}
There is the following commutative diagram of maps induced by the identity
\begin{equation*}
\begin{tikzcd}
{}&{L^2\bar{H}^{\bullet,\bullet}_{BC,sw}}\arrow[dl]\arrow[d]\arrow[dr] &{}\\
{L^2\bar{H}^{\bullet,\bullet}_{\del,sw}}\arrow[dr]&{L^2\bar{H}^{\bullet}_{dR,sw}}\arrow[d] &{L^2\bar{H}^{\bullet,\bullet}_{\delbar,sw}}\arrow[dl]\\
{}&{L^2\bar{H}^{\bullet,\bullet}_{A,sw}}&{}
\end{tikzcd}
\end{equation*}
\end{proposition}
\begin{proof}
The maps $L^2\bar{H}^{\bullet,\bullet}_{BC,sw}\to L^2\bar{H}^{\bullet,\bullet}_{\del,sw}$ and $L^2\bar{H}^{\bullet,\bullet}_{BC,sw}\to L^2\bar{H}^{\bullet,\bullet}_{\delbar,sw}$ are well-defined since for every bidegree $(p,q)$ the equality 
$$\ker(\del+\delbar)_w=\ker\del_w\cap\ker\delbar_w$$ holds in $L^2 \Lambda^{p,q}$ by Lemma \ref{lemma intersection kernel weak strong}, while $$\c{\im\del\delbar_s}\subseteq\c{\im\del_s}\cap\c{\im\delbar_s}$$ holds in $L^2 \Lambda^{p,q}$ by Lemma \ref{lemma image composition strong}.

The maps $L^2\bar{H}^{\bullet,\bullet}_{\del,sw}\to L^2\bar{H}^{\bullet,\bullet}_{A,sw}$ and $L^2\bar{H}^{\bullet,\bullet}_{\delbar,sw}\to L^2\bar{H}^{\bullet,\bullet}_{A,sw}$ are well-defined since for every bidegree $(p,q)$ the inclusion
$$\ker\del_w\cup \ker\delbar_w\subseteq\ker\del\delbar_w$$ holds in $L^2 \Lambda^{p,q}$ by Lemma \ref{lemma kernel composition weak}, while $\c{\im\del_s}\cup\c{\im\delbar_s}\subseteq\c{\im \del_s}+\c{\im\delbar_s}$ always holds.

The map $L^2\bar{H}^{\bullet,\bullet}_{BC,sw}\to L^2\bar{H}^{\bullet}_{dR,sw}$ is well-defined since for every bidegree $(p,q)$ the inclusions $$\ker(\del+\delbar)_w\subseteq\ker d_w,$$ $$\c{\im\del\delbar_s}\subseteq\c{\im d_s}$$ both hold in $L^2 \Lambda^{p,q}$ by Lemma \ref{lemma image composition strong}.

Finally the map $L^2\bar{H}^{\bullet}_{dR,sw}\to L^2\bar{H}^{\bullet,\bullet}_{A,sw}$ is well-defined since for every degree $k$ the inclusion $$\ker d_w\subseteq\bigoplus_{p+q=k}\ker\del\delbar_w\cap L^2\Lambda^{p,q}$$ holds in $L^2 \Lambda^k_{\C}$ by Lemma \ref{lemma kernel composition weak}, while $$\c{\im d_s}\subseteq\bigoplus_{p+q=k}\c{\im(\delbar\oplus\del)_s}\cap L^2\Lambda^{p,q}$$ holds in $L^2 \Lambda^k_{\C}$ by Lemma \ref{lemma closure strong}.
\end{proof}

\begin{remark}
The diagram can be generalised replacing the subscripts $_{sw}$ in the first two lines with $_{sb}$, and replacing $_{sw}$ in the last line with $_{aw}$, for $a,b\in\{s,w\}$. Further generalisations would need Lemmas \ref{lemma image composition strong} and \ref{lemma kernel composition weak} to be generalised as well.
\end{remark}

Now we observe under which conditions the maps in Proposition \ref{prop diagram maps cohomology hilbert} are isomorphisms. 

\begin{proposition}\label{prop equiv cond isom maps diagram}
The maps in Proposition \ref{prop diagram maps cohomology hilbert} are all isomorphisms if and only if all the following equalities hold in $L^2\Lambda^\bullet_\C$
\begin{enumerate}[a)]
\item $\c{\im\del\delbar_s}=\ker\del_w\cap\ker\delbar_w\cap\c{\im d_s}$,
\item $\c{\im\del\delbar_s}=\ker\del_w\cap\c{\im\delbar_s}$,
\item $\c{\im\del\delbar_s}=\ker\del_w\cap\ker\delbar_w\cap\left(\c{\im\del_s}+\c{\im\delbar_s}\right)$,
\item $\ker\del\delbar_w=\ker d_w+\c{\im\del_s}+\c{\im\delbar_s}$,
\item $\ker\del\delbar_w=\ker\del_w+\c{\im\delbar_s}$,
\item $\ker\del\delbar_w=\ker\del_w\cap\ker\delbar_w+\c{\im\del_s}+\c{\im\delbar_s}$.
\end{enumerate}
\end{proposition}
Note that \emph{f)} implies \emph{e)}, therefore the above statement still holds if \emph{e)} is omitted.
\begin{proof}
First of all, note that the inclusions $\subseteq$ in conditions \emph{a)}-\emph{f)} hold either by Lemma \ref{lemma image composition strong} or Lemma \ref{lemma kernel composition weak}. Therefore:

injectivity of $L^2\bar{H}^{\bullet,\bullet}_{BC,sw}\to L^2\bar{H}^{\bullet}_{dR,sw}$ is equivalent to \emph{a)};

injectivity of $L^2\bar{H}^{\bullet,\bullet}_{BC,sw}\to L^2\bar{H}^{\bullet,\bullet}_{\delbar,sw}$ is equivalent to \emph{b)};

injectivity of $L^2\bar{H}^{\bullet,\bullet}_{BC,sw}\to L^2\bar{H}^{\bullet,\bullet}_{\del,sw}$ is equivalent to the conjugate of \emph{b)};

injectivity of $L^2\bar{H}^{\bullet,\bullet}_{BC,sw}\to L^2\bar{H}^{\bullet,\bullet}_{A,sw}$ is equivalent to \emph{c)}, while surjectivity is equivalent to \emph{f)};

surjectivity of $L^2\bar{H}^{\bullet}_{dR,sw}\to L^2\bar{H}^{\bullet,\bullet}_{A,sw}$ is equivalent to \emph{d)};

surjectivity of both maps $L^2\bar{H}^{\bullet,\bullet}_{\del,sw}\to L^2\bar{H}^{\bullet,\bullet}_{A,sw}$ and $L^2\bar{H}^{\bullet,\bullet}_{\delbar,sw}\to L^2\bar{H}^{\bullet,\bullet}_{A,sw}$ follows from \emph{f)}.

To conclude the proof, note that if all the conditions \emph{a)}-\emph{f)} hold, then every map starting from $L^2\bar{H}^{\bullet,\bullet}_{BC,sw}$ is injective and every map arriving in $L^2\bar{H}^{\bullet,\bullet}_{A,sw}$ is surjective. Since the map $L^2\bar{H}^{\bullet,\bullet}_{BC,sw}\to L^2\bar{H}^{\bullet,\bullet}_{A,sw}$ is an isomorphism and the diagram commutes, it follows that every map in the diagram is an isomorphism.
\end{proof}

The proof of Proposition \ref{prop equiv cond isom maps diagram} is purely a matter of linear algebra. In fact it is the analogue result of \cite[Remark 5.16]{DGMS} involving $\del,\delbar,\del\delbar$ on smooth forms. Furthermore, in the smooth case the conditions \emph{a)}-\emph{f)} are all equivalent to each other \cite[Lemma 5.15]{DGMS} (see Section \ref{sec complex manifold}). Conversely, here we do not see any simple reason for which the conditions \emph{a)}-\emph{f)} in Proposition \ref{prop equiv cond isom maps diagram} should be equivalent. It would be interesting to know if this is the case or to exhibit counterexamples. We will see in the next section that, on complete K\"ahler manifolds, all the conditions in Proposition \ref{prop equiv cond isom maps diagram} hold.

\section{Complete K\"ahler manifolds}\label{sec complete kahler}
Let $(M,g)$ be a complete K\"ahler manifold. The K\"ahler assumption guarantees that the relations \eqref{eq laplacians 2order kahler}, \eqref{eq laplacian bc kahler} and \eqref{eq laplacian a kahler} all hold. Meanwhile, the completeness of the metric implies that $\delta_s=\delta_w$ for $\delta=d,d^*,\del,\del^*,\delbar,\delbar^*,\delbar\oplus\del,\del^*\oplus\delbar^*$; we refer to \cite[Theorem 1.3]{A} for a proof of this. Throughout this section, we will write either $\delta_s$ or $\delta_w$ depending on which is more convenient at the time.

Completeness also implies that all positive integer powers of $(d+d^*)_0,(\del+\del^*)_0,(\delbar+\delbar^*)_0$ as operators $A^{\bullet,\bullet}_0\to A^{\bullet,\bullet}_0$ are essentially self-adjoint \cite[Section 3.B]{Che}. For example, $\Delta_{\delbar,sw}$ is the unique self-adjoint extension of $(\Delta_\delbar)_0=(\delbar+\delbar^*)_0^2$. Similarly the operator $\Delta_{\delbar,sw}^k$, defined via the Spectral Theorem \ref{thm spectral unbounded self-adjoint}, is the unique self-adjoint extension of $(\Delta_\delbar)_0^k=(\delbar+\delbar^*)_0^{2k}$. Analogously $(\Delta_{\del,sw})^k$ is the unique self-adjoint extension of $(\Delta_\del)_0^k=(\del+\del^*)_0^{2k}$ and $\Delta_{d,sw}$ is the unique self-adjoint extension of $(\Delta_d)_0^k=(d+d^*)_0^{2k}$.
Finally, note that relation \eqref{eq laplacians 2order kahler}, together with the essential self-adjointness of the above Laplacians, implies 
\begin{equation}\label{eq 2 order laplacians complete kahler}
\Delta_{d,sw}=2\Delta_{\del,sw}=2\Delta_{\delbar,sw}.
\end{equation}

\subsection{Equality between \texorpdfstring{$L^2$}{L2} Dolbeault and Bott-Chern harmonic forms}\,\\
On compact K\"ahler manifolds, it is known that the kernel of the Dolbeault Laplacian coincides with the kernels of the Bott-Chern and Aeppli Laplacians. We will now show that this statement can be generalised to the complete K\"ahler case.

We begin by giving a characterisation of the space of $L^2$ Dolbeault harmonic forms.
\begin{lemma}\label{lemma char ker dolb kahl compl}
Let $(M,g)$ be a complete K\"ahler manifold. Then
\begin{equation*}
L^2\H^{p,q}_{\delbar,sw}=\{\alpha\in L^2 A^{p,q}\,|\,\delbar\alpha=\delbar^*\alpha=\del\alpha=\del^*\alpha=0\}.
\end{equation*}
\end{lemma}
\begin{proof}
Note that \eqref{eq 2 order laplacians complete kahler}  and \eqref{eq kernel dolbeault laplacian ab} imply
\[
L^2\H^{p,q}_{\delbar,sw}=\ker\delbar_w\cap\ker\delbar^*_w\cap\ker\del_w\cap\ker\del^*_w\cap L^2\Lambda^{p,q},
\]
which is a space of smooth forms by Theorem \ref{thm ell-reg}. From Lemma \ref{lemma kernel weak smooth} we conclude the proof.
\end{proof}

On a general Hermitian manifold $(M,g)$ we define extensions of the Bott-Chern and Aeppli Laplacians $(\tilde\Delta_{BC})_0$ and $(\tilde\Delta_{A})_0$:
\begin{align*}
\tilde\Delta_{BC,bc}:=&\del\delbar_{b}\delbar^*\del^*_{b'}+
\delbar^*\del^*_s\del\delbar_w+
\del^*\delbar_s\delbar^*\del_w+\delbar^*\del_s\del^*\delbar_w+
(\del^*\oplus\delbar^*)_{c'}(\del+\delbar)_c,\\
\tilde\Delta_{A,ab}:=& \delbar^*\del^*_{b'}\del\delbar_{b}+
\del\delbar_s\delbar^*\del^*_w+\delbar\del^*_s\del\delbar^*_w+ \del\delbar^*_s\delbar\del^*_w +
(\delbar \oplus \del)_a(\del^*+\delbar^*)_{a'},
\end{align*}
for $a,b,c\in\{s,w\}$.
They are positive and self-adjoint on the domains given by Theorems \ref{thm composition with adjoint} and \ref{thm sum of self ajoint}.
Furthermore, they have the same kernels as the operators $\Delta_{BC,bc}$ and $\Delta_{A,ab}$, respectively.

\begin{lemma}\label{lemma kernel bc}
Let $(M,g)$ be a Hermitian manifold. For $a,b,c\in\{s,w\}$
\begin{align*}
\ker\tilde\Delta_{BC,bc}&=\ker \delbar^* \del^*_{b'}\cap\ker (\del+\delbar)_c=L^2\H^{p,q}_{BC,bc},\\
\ker\tilde\Delta_{A,ab}&=\ker \del \delbar_{b}\cap\ker (\del^*+\delbar^*)_{a'}=L^2\H^{p,q}_{A,ab}.
\end{align*}
\end{lemma}
\begin{proof}
By Lemma \ref{lemma ker lapl} we deduce that in $L^2\Lambda^{p,q}$ we have
\[
\ker \tilde{\Delta}_{BC,bc} = \ker \delbar^* \del^*_{b'} \cap \ker \del \delbar_w \cap \ker \delbar^* \del_w \cap \ker \del^* \delbar_w \cap \ker (\del+\delbar)_c.
\]
The desired result then follows from Lemma \ref{lemma intersection kernel weak strong}, since
\[
\ker (\del+\delbar)_c\subseteq\ker\del_c\cap\ker\delbar_c\subseteq\ker\delbar_w\cap\ker\del_w,
\]
and from Lemma \ref{lemma kernel composition weak}, noting that
\[
\ker\delbar_w \subseteq\ker \del \delbar_w \cap \ker \del^* \delbar_w,\ \ \ \ker\del_w\subseteq\ker \delbar^* \del_w.
\]
The Aeppli case is analogous.
\end{proof}

\begin{remark}\label{rmk kernel bc a complete}
If $(M,g)$ is a complete Hermitian manifold, then $(\del+\delbar)_c=(\del+\delbar)_w$ and $(\del^*+\delbar^*)_{a'}=(\del^*+\delbar^*)_w$ by completeness. By Lemma \ref{lemma intersection kernel weak strong} we also have $\ker(\del+\delbar)_w=\ker\del_w\cap\ker\delbar_w$ and $\ker(\del^*+\delbar^*)_w=\ker\del^*_w\cap\ker\delbar^*_w$. Therefore, we can write
\[
\ker (\del+\delbar)_c=\ker\del_w\cap\ker\delbar_w,\ \ \ \ker(\del^*+\delbar^*)_{a'}=\ker\del^*_w\cap\ker\delbar^*_w.
\]
\end{remark}


Consider the fourth order terms appearing in the operators $\tilde\Delta_{BC,bc}$ and $\tilde\Delta_{A,ab}$. Taken without the accompanying second order terms, they define the following positive self-adjoint operators:
\begin{align*}
\tilde\Delta_{BC,b,4}:=&\del\delbar_{b}\delbar^*\del^*_{b'}+
\delbar^*\del^*_s\del\delbar_w+
\del^*\delbar_s\delbar^*\del_w+\delbar^*\del_s\del^*\delbar_w,\\
\tilde\Delta_{A,b,4}:=& \delbar^*\del^*_{b'}\del\delbar_{b}+
\del\delbar_s\delbar^*\del^*_w+\delbar\del^*_s\del\delbar^*_w+ \del\delbar^*_s\delbar\del^*_w,
\end{align*}
with the domains given by Theorems \ref{thm composition with adjoint} and \ref{thm sum of self ajoint}.

If we assume the Hermitian manifold $(M,g)$ is K\"ahler and complete then, by \eqref{eq laplacian bc kahler} and \eqref{eq laplacian a kahler}, it follows that $\tilde\Delta_{BC,b,4}$ and $\tilde\Delta_{A,b,4}$ are extensions of $(\Delta_\delbar)^2_0$. Since $(\Delta_\delbar)^2_0$ is essentially self-adjoint, we then deduce that
\[
\tilde\Delta_{BC,b,4}=\tilde\Delta_{A,b,4}=\Delta_{\delbar,sw}^2.
\]
By Lemma \ref{lemma ker lapl}, note that in $L^2\Lambda^{p,q}$
\begin{align*}
\ker\tilde\Delta_{BC,b,4}&=\ker\delbar^*\del^*_{b'}\cap\ker\del\delbar_w\cap\ker\delbar^*\del_w\cap\ker\del^*\delbar_w,\\
\ker\tilde\Delta_{A,b,4}&=\ker\del\delbar_{b}\cap\ker\delbar^*\del^*_w\cap\ker\del\delbar^*_w\cap\ker\delbar\del^*_w,\\
\ker\Delta_{\delbar,sw}^2&=\ker\Delta_{\delbar,sw}.
\end{align*}

Now we can prove that the kernel of the unique self-adjoint extension of the Dolbeault Laplacian coincides with the kernels of our self-adjoint extensions of the Bott-Chern and Aeppli Laplacians.

\begin{theorem}\label{thm kahler complete kernel}
Let $(M,g)$ be a complete K\"ahler manifold. Then for $a,b,c\in\{s,w\}$ it holds that
\[
L^2\H^{p,q}_{BC,bc}=L^2\H^{p,q}_{A,ab}=L^2\H^{p,q}_{\delbar,sw}.
\]
In particular, the corresponding reduced cohomology spaces are isomorphic  $$L^2\bar{H}^{p,q}_{BC,bc}\simeq L^2\bar{H}^{p,q}_{A,ab}\simeq L^2\bar{H}^{p,q}_{\delbar,sw}.$$
\end{theorem}
\begin{proof}
Let us first prove $\ker\tilde\Delta_{BC,bc}\supseteq\ker\Delta_{\delbar,sw}$. By Lemma \ref{lemma kernel bc} and Remark \ref{rmk kernel bc a complete}, this is the same as proving
\[
\ker\delbar^*\del^*_{b'}\cap\ker\del_w\cap\ker\delbar_w\supseteq\ker\Delta_{\delbar,sw}.
\]
 By Lemma \ref{lemma char ker dolb kahl compl} and Lemma \ref{lemma kernel weak smooth} we get $\ker\del_w\cap\ker\delbar_w\supseteq\ker\Delta_{\delbar,sw}$, while $\ker\delbar^*\del^*_{b'}\supseteq\ker\Delta_{\delbar,sw}$ follows from $\ker\Delta_{\delbar,sw}=\ker\Delta_{\delbar,sw}^2=\ker\tilde\Delta_{BC,b,4}$.

Conversely, let $\alpha\in\ker\tilde\Delta_{BC,bc}$; in particular $\alpha\in\ker\del_w\cap\ker\delbar_w$ by Remark \ref{rmk kernel bc a complete}. Then, for every $\gamma\in A^{\bullet,\bullet}_0$, using \eqref{eq laplacian bc kahler}
\begin{align*}
0=\la\tilde\Delta_{BC,bc}\alpha,\gamma\ra&=\la\alpha,\tilde\Delta_{BC}\gamma\ra,\\
&=\la\alpha,\Delta_\delbar^2\gamma+\del^*\del\gamma+\delbar^*\delbar\gamma\ra,\\
&=\la\alpha,\Delta_\delbar^2\gamma\ra,
\end{align*}
therefore $\alpha\in\ker(\Delta_\delbar^2)_w$. Since $(\Delta_\delbar^2)_0$ is essentially self-adjoint, then $(\Delta_\delbar^2)_w=\Delta_{\delbar,sw}^2$ and so $\ker\Delta_{\delbar,sw}=\ker\Delta_{\delbar,sw}^2$, which allows us to conclude.

In the same way we can prove that $\ker\tilde\Delta_{A,ab}=\ker\Delta_{\delbar,sw}$.
\end{proof}

\begin{corollary}\label{cor isom reduced cohom diagram}
If $(M,g)$ is a complete K\"ahler manifold, then the maps between the $L^2$ reduced cohomology spaces in Proposition \ref{prop diagram maps cohomology hilbert} are all isomorphisms, and all the conditions a),\dots,f) in Proposition \ref{prop equiv cond isom maps diagram} hold.
\end{corollary}
\begin{proof}
By Theorem \ref{thm kahler complete kernel} and \eqref{eq 2 order laplacians complete kahler}, each class of the reduced $L^2$ cohomology spaces in Proposition \ref{prop diagram maps cohomology hilbert} have a unique representative in $L^2\H^{\bullet}_{d,sw}$. Since all the maps are induced by the identity, it follows that they are all isomorphisms. By Proposition \ref{prop equiv cond isom maps diagram}, we get the last claim.
\end{proof}

An immediate consequence of conditions \emph{a),\dots,f)} is the following.

\begin{corollary}[$L^2$ reduced $\del\delbar$-Lemma on complete K\"ahler manifolds]\label{cor l2 deldelbar lemma}
Let $(M,g)$ be a complete K\"ahler manifold. If $\alpha\in L^2\Lambda^{\bullet}_\C\cap \ker\del_w\cap\ker\delbar_w$, then
\[
\alpha\in\c{\im\del\delbar_s}\iff\alpha\in\c{\im d_s} \iff\alpha\in\c{\im\del_s}\iff\alpha\in\c{\im\delbar_s}\iff \alpha\in\c{\im\delbar_s} +\c{\im\del_s}.
\]
\end{corollary}

\begin{corollary}\label{cor compl kahl ker im deldelbar}
Let $(M,g)$ be a complete K\"ahler manifold. Then
\[
\ker\del\delbar_s=\ker\del\delbar_w,\ \ \ \c{\im\del\delbar_s}=\c{\im\del\delbar_w},
\]
\[
\ker\delbar^*\del^*_s=\ker\delbar^*\del^*_w,\ \ \ \c{\im\delbar^*\del^*_s}=\c{\im\delbar^*\del^*_w}.
\]
\end{corollary}
\begin{proof}
By Theorem \ref{thm l2 bc decomp}, Remark \ref{rmk kernel bc a complete} and Theorem \ref{thm kahler complete kernel} we have
\begin{align*}
&\ker\del_w\cap\ker\delbar_w=L^2\H^{p,q}_{\delbar,sw}\oplus\c{\im\del\delbar_a},\\
&\ker\del\delbar_b=L^2\H^{p,q}_{\delbar,sw}\oplus(\c{\im\del_s}+\c{\im\delbar_s}).
\end{align*}
For different choices of $a,b\in\{s,w\}$, the thesis follows.
\end{proof}

\subsection{Spectral gap of the Dolbeault Laplacian}\,\\We will now study the case of a complete K\"ahler manifold for which $\Delta_{\delbar,sw}$ has a spectral gap in $L^2\Lambda^{\bullet,\bullet}$. Specifically, we will explore the consequences this has for $L^2$ Bott-Chern and Aeppli cohomologies. First of all, note that in this case by Lemma \ref{lemma spectral gap equiv 2}, $\im\delbar_s$, $\im\delbar^*_s$, and $\im \Delta_{\delbar,sw}$ are closed. Second, since $\Delta_{d,sw}=2\Delta_{\del,sw}=2\Delta_{\delbar,sw}$ by \eqref{eq 2 order laplacians complete kahler}, it follows that $\im\del_s$, $\im d_s$, $\im\del^*_s$ and $\im d^*_s$  are also closed.
In particular, if $\Delta_{\delbar,sw}$ has a spectral gap in $L^2\Lambda^{\bullet,\bullet}$, then the reduced and unreduced $L^2$ cohomologies coincide for both the Dolbeault and de Rham cases, \textit{i.e.}, for $a,b\in\{s,w\}$ with $a\le b$, we have 
\[
L^2\bar{H}^{\bullet,\bullet}_{\delbar,ab}=L^2{H}^{\bullet,\bullet}_{\delbar,ab},\ \ \ 
L^2\bar{H}^{\bullet,\bullet}_{\del,ab}=L^2{H}^{\bullet,\bullet}_{\del,ab},\ \ \ L^2\bar{H}^{\bullet}_{d,ab}=L^2{H}^{\bullet}_{d,ab}.
\]

We want to conclude that the same holds for $L^2$ Bott-Chern and Aeppli cohomology. In proving this, we will also prove that a spectral gap of $\Delta_{\delbar,sw}$ implies a spectral gap of $\Delta_{A,ab}$ and $\Delta_{BC,ab}$. A simple proof of this result does not seem viable and so, over the course of the next couple pages, we will work our way towards it, primarily through the use of Lemma \ref{lemma spectral gap equiv 2}.

We start by proving that $\im(\delbar \oplus\del)_s$ and $\im(\del^*\oplus\delbar^*)_s$ are closed in this setting.

\begin{theorem}\label{thm closedness delbar+del}
Let $(M,g)$ be a complete K\"ahler manifold. If $\Delta_{\delbar,sw}$ has a spectral gap in $L^2\Lambda^{p,q-1}\oplus L^2\Lambda^{p-1,q}$, then $\im(\delbar \oplus\del)_s$ is closed in $L^2\Lambda^{p,q}$. If $\Delta_{\delbar,sw}$ has a spectral gap in $L^2\Lambda^{p+1,q}\oplus L^2\Lambda^{p,q+1}$, then $\im(\del^*\oplus\delbar^*)_s$ is closed in $L^2\Lambda^{p,q}$.
\end{theorem}
\begin{proof}
We just prove the first claim since the second one is analogous. Our strategy is to use Lemma \ref{lemma spectral gap equiv 2}, therefore we consider the following Hilbert complex:
\[
L^2\Lambda^{p,q-2}\oplus L^2\Lambda^{p-1,q-1}\oplus L^2\Lambda^{p-2,q}\overset{(\delbar\oplus d\oplus \del)_s}{\longrightarrow}L^2\Lambda^{p,q-1}\oplus L^2\Lambda^{p-1,q}\overset{(\delbar\oplus\del)_w}{\longrightarrow}L^2\Lambda^{p,q}.
\]
We consider an operator $\square: A^{p,q-1}\oplus A^{p-1,q}\to A^{p,q-1}\oplus A^{p-1,q}$ defined by
\[
\square:=(\del^*+\delbar^*)(\delbar\oplus\del)+(\delbar\oplus d\oplus \del)(\del^*+\delbar^*),
\]
where $\delbar\oplus d\oplus \del$ acts on $A^{p,q-2}\oplus A^{p-1,q-1}\oplus A^{p-2,q}$ component by component. Define also $\square_{sw}:L^2\Lambda^{p,q-1}\oplus L^2\Lambda^{p-1,q}\to L^2\Lambda^{p,q-1}\oplus L^2\Lambda^{p-1,q}$ by
\[
\square_{sw}:=(\del^*+\delbar^*)_s(\delbar\oplus\del)_w+(\delbar\oplus d\oplus \del)_s(\del^*+\delbar^*)_w,
\]
so that $\square_{sw}$ is the Laplacian associated to the above Hilbert complex. By Theorems \ref{thm composition with adjoint} and \ref{thm sum of self ajoint} it is a self-adjoint extension of the operator $\square_0$ on $A^{p,q-1}_0\oplus A^{p-1,q}_0$. By Lemma \ref{lemma spectral gap equiv 2}, to prove the theorem it is sufficient to show that $\square_{sw}$ has a spectral gap.
Note that by the K\"ahler identities, namely using $\del\delbar^*+\delbar^*\del=0$ and $\delbar\del^*+\del^*\delbar=0$, we have
\[
\square(\alpha+\beta)=\Delta_\delbar(\alpha+\beta)+\del\del^*\alpha+\delbar\delbar^*\beta
\]
for all $\alpha\in A^{p,q-1}$ and $\beta\in A^{p-1,q}$.
Arguing as in \cite[Theorem 2.4]{Str}, we can show that $\square_0$ is essentially self-adjoint. Therefore, if by Theorems \ref{thm composition with adjoint} and \ref{thm sum of self ajoint} we define the self-adjoint operator
\[
\square_{sw}':=\Delta_{\delbar,sw}+\del_s\del^*_w\oplus\delbar_s\delbar^*_w:L^2\Lambda^{p,q-1}\oplus L^2\Lambda^{p-1,q}\to L^2\Lambda^{p,q-1}\oplus L^2\Lambda^{p-1,q},
\]
then $\square_{sw}=\square_{sw}'$. We are left to prove that $\square_{sw}'$ has a spectral gap. Note that in $L^2\Lambda^{p,q-1}\oplus L^2\Lambda^{p-1,q}$ it holds that $\D(\square_{sw}')= \D(\Delta_{\delbar,sw})$, and $\ker\square_{sw}'=\ker\Delta_{\delbar,sw}$ by Lemma \ref{lemma char ker dolb kahl compl}. Then there exists a $C>0$ such that for all $(\alpha\oplus\beta)\in\D(\square_{sw}')\cap(\ker\square_{sw}')^\perp$
\begin{align*}
C\lv\alpha+\beta\rv^2&\le\la \Delta_{\delbar,sw}(\alpha+\beta),\alpha+\beta\ra\\
&\le\la \Delta_{\delbar,sw}(\alpha+\beta),\alpha+\beta\ra+\lv\del^*_w\alpha\rv^2+\lv\delbar^*_w\beta\rv^2\\
&=\la \Delta_{\delbar,sw}(\alpha+\beta),\alpha+\beta\ra+\la\del_s\del^*_w\alpha,\alpha\ra+\la\delbar_s\delbar^*_w\beta,\beta\ra\\
&=\la \Delta_{\delbar,sw}(\alpha+\beta)+\del_s\del^*_w\alpha+\delbar_s\delbar^*_w\beta,\alpha+\beta\ra\\
&=\la\square_{sw}'(\alpha+\beta),\alpha+\beta\ra,
\end{align*}
namely $\square_{sw}'$ has a spectral gap.
\end{proof}

\begin{remark}
The Hilbert complex considered in the previous proof is nothing but the preceding stage of the $L^2$ ABC Hilbert complex as shown in \eqref{eq bc hilbert complex} (see Section \ref{sec questions} for more details).
\end{remark}

Now we deal with the closedness of $\im\del\delbar_s$ and $\im\delbar^*\del^*_s$.
By the Spectral Theorem \ref{thm spectral unbounded self-adjoint}, we see that if $\Delta_{\delbar,sw}$ has a spectral gap in $L^2\Lambda^{p,q}$, then $\Delta_{\delbar,sw}^k$ also has a spectral gap in $L^2\Lambda^{p,q}$ for all $k\in\N$, $k\ge2$. Namely, if $C=\inf(\sigma(\Delta_{\delbar,sw})\setminus\{0\})>0$, then $C^k=\inf(\sigma(\Delta_{\delbar,sw}^k)\setminus\{0\})>0$. In particular, for $b\in\{s,w\}$, since $\Delta_{\delbar,sw}^2=\tilde\Delta_{BC,b,4}=\tilde\Delta_{A,b,4}$, then $\tilde\Delta_{BC,b,4}$ and $\tilde\Delta_{A,b,4}$ both also have a spectral gap.
The next theorem, the proof of which is inspired by \cite[Theorem 1.4.A]{G}, shows that $\im\del\delbar_s$ and $\im\delbar^*\del^*_s$ are closed in this setting.

\begin{theorem}\label{thm spectral gap closed image deldelbar}
Let $(M,g)$ be a complete K\"ahler manifold. If $\Delta_{\delbar,sw}$ has a spectral gap in $L^2\Lambda^{p,q}$, then  $\im\del\delbar_b$ is closed in $L^2\Lambda^{p+1,q+1}$ and $\im\delbar^*\del^*_b$ is closed in $L^2\Lambda^{p-1,q-1}$ for $b\in\{s,w\}$.
\end{theorem}
\begin{proof}
By the previous discussion and Lemma \ref{lemma spectral gap equiv}, since $\tilde\Delta_{A,w,4}$ has a spectral gap, there is a constant $C>0$ such that
\begin{equation}\label{eq spectral gap dolb}
\lv\del\delbar_w\psi\rv^2+\lv\delbar^*\del^*_w\psi\rv^2+\lv\del\delbar^*_w\psi\rv^2+\lv\delbar\del^*_w\psi\rv^2=\la \psi,\tilde\Delta_{A,w,4}\psi\ra\ge C \la\psi,\psi\ra
\end{equation}
for all $\psi\in\D(\tilde\Delta_{A,w,4})\cap (\ker\tilde\Delta_{A,w,4})^\perp\cap L^2\Lambda^{p,q}$. 

Recall that $\c{\im\del\delbar_b}=\c{\im\del\delbar_s}$ by Corollary \ref{cor compl kahl ker im deldelbar}.
Let $\alpha\in\c{\im\del\delbar_s}\cap L^2\Lambda^{p+1,q+1}=\c{\del\delbar A^{p,q}_0}$, then there is a sequence of $\beta_k\in A^{p,q}_0$ such that $\del\delbar\beta_k\to\alpha$. By the Aeppli decompositions of Theorem \ref{thm l2 bc decomp}, we obtain
\[
L^2A^{p,q}=\ker\del\delbar_w\cap A^{p,q}\oplus\c{\im\delbar^*\del^*_s}\cap A^{p,q},
\]
therefore we can decompose
\[
\beta_k=\eta_k+\theta_k,
\]
where $\eta_k\in\ker\del\delbar_w\cap A^{p,q}$ and $\theta_k\in\c{\im\delbar^*\del^*_s}\cap A^{p,q}$. In particular $\theta_k\in(\ker\tilde\Delta_{A,w,4})^\perp$ and $\del\delbar\beta_k=\del\delbar\eta_k+\del\delbar\theta_k=\del\delbar\theta_k$ by Lemma \ref{lemma kernel weak smooth}. By Lemma \ref{lemma pq0 subset}, we also know that $\theta_k\in\ker\delbar^*_w\cap\ker\del^*_w$, therefore $\delbar^*\theta_k=\del^*\theta_k=0$ by Lemma \ref{lemma kernel weak smooth}, implying that $\theta_k\in\D(\tilde\Delta_{A,sw,4})$.

Applying \eqref{eq spectral gap dolb} we find that
\[
C\lv\theta_k-\theta_j\rv^2\le \lv\del\delbar\theta_k-\del\delbar\theta_j\rv^2\to0
\]
as $k,j\to+\infty$ since the sequence $\del\delbar\beta_k=\del\delbar\theta_k$ is Cauchy.
From this we conclude that the sequence $\theta_k$ is also Cauchy and so it must converge to some form $\theta\in L^2\Lambda^{p,q}$, proving that $\del\delbar_s\theta=\alpha$. In particular $\del\delbar_b\theta=\alpha$.

The proof for $\im\delbar^*\del^*_b$ is analogous, using the Bott-Chern decompositions of Theorem \ref{thm l2 bc decomp} and the spectral gap property for $\tilde\Delta_{BC,s,4}$.
\end{proof}

\begin{corollary}
Let $(M,g)$ be a complete K\"ahler manifold. If $\Delta_{\delbar,sw}$ has a spectral gap in $L^2\Lambda^{\bullet,\bullet}$, then for $a,b\in\{s,w\}$ with $a\le b$
\[
L^2\bar{H}^{\bullet,\bullet}_{BC,ab}=L^2{H}^{\bullet,\bullet}_{BC,ab},\ \ \ L^2\bar{H}^{\bullet,\bullet}_{A,ab}=L^2{H}^{\bullet,\bullet}_{A,ab}.
\]
\end{corollary}

Applying Lemma \ref{lemma spectral gap equiv 2}, we finally see that a spectral gap of $\Delta_{\delbar,sw}$ actually implies a spectral gap also of $\Delta_{A,ab}$ and $\Delta_{BC,ab}$.
\begin{corollary}
Let $(M,g)$ be a complete K\"ahler manifold. If $\Delta_{\delbar,sw}$ has a spectral gap in $L^2\Lambda^{\bullet,\bullet}$, then $\Delta_{A,ab}$ and $\Delta_{BC,ab}$ have a spectral gap in $L^2\Lambda^{\bullet,\bullet}$, with $a,b\in\{s,w\}$ and $a\le b$.
\end{corollary}

\begin{corollary}
Let $(M,g)$ be a complete K\"ahler manifold. If $\Delta_{\delbar,sw}$ has a spectral gap in $L^2\Lambda^{\bullet,\bullet}$, then
\[
\del\delbar_s=\del\delbar_w,\ \ \ \delbar^*\del^*_s=\delbar^*\del^*_w.
\]
\end{corollary}
\begin{proof}
By \cite[Proposition 1.6]{B1}, if $\ker\del\delbar_s=\ker\del\delbar_w$ and $\im\del\delbar_s=\im\del\delbar_w$, then $\del\delbar_s=\del\delbar_w$. This follows from Corollary \ref{cor compl kahl ker im deldelbar} and Theorem \ref{thm spectral gap closed image deldelbar}. The same argument shows that we also have $\delbar^*\del^*_s=\delbar^*\del^*_w$.
\end{proof}
In particular, if on a complete K\"ahler manifold $\Delta_{\delbar,sw}$ has a spectral gap in $L^2\Lambda^{\bullet,\bullet}$, then there is a unique $L^2$ ABC Hilbert complex.

We end the section recalling a sufficient condition for $\Delta_{\delbar,sw}$ to have a spectral gap, due to Gromov. 
\begin{theorem}[{\cite[Theorems 1.2.B, 1.4.A]{G}}]
Let $(M,g)$ be a complete K\"ahler manifold of complex dimension $n$. If the fundamental form $\omega$ is $d$-bounded, \textit{i.e.}, $\omega=d\eta$ with $\eta$ bounded, then $\ker\Delta_{\delbar,sw}\cap L^2\Lambda^{p,q}=\{0\}$ if $p+q\ne n$ and $\Delta_{\delbar,sw}$ has a spectral gap in $L^2\Lambda^{\bullet,\bullet}$.
\end{theorem}

\section{Galois coverings of a compact complex manifold}\label{sec coverings}
Given a smooth manifold $M$, we say that two Riemannian metrics $g^{(1)}$ and $g^{(2)}$ on $M$ are \emph{quasi-isometric} if there exists a positive real constant $C$ such that at every point $p \in M$ and for every tangent vector $X_p \in T_pM$ we have
\begin{equation}\label{quasi-isom}
    \frac 1C g^{(2)}_p(X_p,X_p) \leq g^{(1)}_p(X_p,X_p) \leq C g^{(2)}_p(X_p,X_p). 
\end{equation}

Now let $M$ be a complex manifold. As outlined in Section \ref{sec complex manifold}, any Hermitian metric on $M$ induces a Hermitian metric on the bundle of $(p,q)$-forms $\Lambda^{p,q}M$. Furthermore, if we have two quasi-isometric Hermitian metrics $g^{(1)}$ and $g^{(2)}$ on $M$, the induced metrics $h^{(1)}$ and $h^{(2)}$ will satisfy the same inequality, \textit{i.e.}, there exists a constant such that at any point $p \in M$ a general $(p,q)$-form $\alpha_p \in \Lambda^{p,q}_p M$ satisfies 
$$\frac 1C h_p^{(2)}(\alpha_p,\alpha_p) \leq h_p^{(1)}(\alpha_p,\alpha_p) \leq C h_p^{(2)}(\alpha_p,\alpha_p).  $$
This pointwise inequality implies a more general inequality for the $L^2$ norm. Specifically, for some constant $C>0$, any (not necessarily smooth) $(p,q)$-form $\alpha:M \rightarrow \Lambda^{p,q}M$ satisfies
$$\frac 1C \| \alpha\|_{L^2,h^{(2)}} \leq \| \alpha \|_{L^2, h^{(1)}} \leq C \| \alpha \|_{L^2, h^{(2)}}. $$ 
From this we see that, although quasi-isometry may change the norm of $L^2 \Lambda^{p,q}$, the underlying set of $L^2$ $(p,q)$-forms does not change. Furthermore, quasi-isometry preserves the notion of convergence: if a sequence within $L^2 \Lambda^{p,q}$ converges for one Hermitian metric, it converges for all metrics in the same quasi-isometry class. It follows that if $P$ is a differential operator that doesn't depend on the metric (\textit{e.g.}, $d,\del,\delbar,\del\delbar$) the strong extension $P_s$ defined with respect to two quasi-isometric metrics is the same operator. It is well known that this same result holds also for the weak extension $P_w$; it follows from a metric-independent generalisation of the notion of the formal adjoint, which will not be treated here. However, we will see later in this section that, in the setting we are interested in, $P_s=P_w$ for $P\in\{d,\del,\delbar,\del\delbar\}$, therefore the previous discussion for $P_s$ is sufficient to show that also $P_w$ depends only on the quasi-isometry class of the metric.

From this we can obtain the following result for Bott-Chern and Aeppli cohomology.
\begin{proposition}\label{dependence of cohom}
On a Hermitian manifold, the reduced and unreduced $L^2$ Bott-Chern cohomologies, $L^2 \bar{H}_{BC,bc}^{p,q}$ and $L^2 H_{BC,bc}^{p,q}$, and also the reduced and unreduced $L^2$ Aeppli cohomologies, $L^2 \bar{H}_{A,ab}^{p,q}$ and $L^2 H_{A,ab}^{p,q}$, depend only on the quasi-isometry class of the Hermitian metric and the choice of closed extensions $a,b,c \in \{s,w\}$ with $a\leq b\leq c$ in the $L^2$ Aeppli-Bott-Chern complex.
\end{proposition}
By the same argument it is well known that the above proposition holds for other $L^2$ cohomologies defined on differential forms, such as the de Rham or Dolbeault cohomology.

Following from the discussion after Theorem \ref{thm isom cohm smooth} we know that, on a compact complex manifold, the $L^2$ Aeppli and Bott-Chern cohomologies are uniquely defined and coincide with the cohomologies defined on smooth forms.  In this section, we consider the more general case of a Galois covering of a compact complex manifold.

Let $\pi: \widetilde{M} \rightarrow M$ be a covering of a smooth manifold $M$ with $\Gamma$ denoting the group of deck transformations. If $\Gamma$ acts transitively on the fibre $\pi^{-1}(p)$ for all points $p \in M$ and $\widetilde{M}$ is connected, then we say that $\pi$ is a \emph{Galois $\Gamma$-covering} and $\widetilde{M}$ is a \emph{Galois $\Gamma$-covering space}. 
One example of a Galois covering is provided by the universal covering.

Conversely, if $\Gamma$ is a discrete group acting freely and properly discontinuously on a smooth connected manifold $\widetilde{M}$, then $\widetilde{M}/\Gamma$ is a smooth manifold and the map
$$\pi: \widetilde{M} \rightarrow \widetilde{M}/\Gamma $$
is a Galois $\Gamma$-covering.

The main theorem of Galois theory on coverings states the following.
Let $M$ be a connected manifold. Given a connected covering $\widetilde{M} \rightarrow M$, we obtain a subgroup $\pi_1(\widetilde M)\subset\pi_1(M)$. This defines a bijection between the isomorphism classes of connected coverings of $M$ and the subgroups $G \subset\pi_1(M)$. Under
this correspondence, a Galois $\Gamma$-covering corresponds to the normal subgroup
$G \subset\pi_1(M)$ for which $\Gamma$ is isomorphic to $\pi_1(M)/G$, and $\widetilde M$ is diffeomorphic to the quotient of the universal covering of $M$ by $G$.

Suppose now that $M$ is a complex manifold. By taking the pullback with respect to the covering map $\pi$ of the complex structure of $M$, a Galois $\Gamma$-covering space $\widetilde{M}$ inherits a complex structure. We will now show that, when $M$ is a compact complex manifold, the $L^2$ Aeppli and Bott-Chern cohomology spaces can be uniquely defined on $\widetilde{M}$ with the induced complex structure. By Proposition \ref{dependence of cohom} this means we must rule out dependence on the Hermitian metric and on the choice of closed extensions in the Hilbert complex.

We shall first consider any dependence the cohomology spaces have on the Hermitian metric. In particular, we will show that the $L^2$ cohomology spaces on $\widetilde M$ are the same for any choice of $\Gamma$-invariant metric on $\widetilde{M}$. A metric on $\widetilde{M}$ is said to be \textit{$\Gamma$-invariant} if it invariant under the pullback with respect to any $\gamma \in \Gamma$, namely if  $g=\gamma^*g$, \textit{i.e.}, if at any point $p \in \widetilde{M}$ we have
    $$g_p(X_p,Y_p) = g_{\gamma(p)}(d_p \gamma(X_{p}), d_p \gamma (Y_p) ) $$
     for all $\gamma \in \Gamma$ and all $X_p, Y_p \in T_p\widetilde{M}$. Note that a metric $\tilde g$ is $\Gamma$-invariant if and only if it is given by the pullback of a metric $g$ on $M$, that is $\tilde g=\pi^* g$. 

We prove the following proposition.
\begin{proposition}
If $M$ is compact, then any two $\Gamma$-invariant Riemannian metrics on a Galois $\Gamma$-covering space $\widetilde{M}$ are quasi-isometric.
\end{proposition}
\begin{proof}
    We first show that any two metrics on a compact manifold are quasi-isometric.

    Let $g^{(1)}, g^{(2)}$ be a pair of Riemannian metrics. We can define the unit sphere bundle $SM \subset TM$ to be the set of unit vectors with respect to the first metric,
    $$SM = \left\{X_p \in T_pM \,\middle|\, g^{(1)}_p(X_p, X_p) = 1,\   p \in M\right\}. $$

   The second metric then defines a function on $SM$, $X_p \mapsto g^{(2)}_p(X_p,X_p)$. Since $SM$ is a compact manifold, this function has a well-defined maximum and minimum
   $$A := \min_{X_p \in SM}g^{(2)}_p(X_p,X_p)\quad \quad \quad B := \max_{X_p \in SM}g_p^{(2)}(X_p,X_p).$$
   Moreover, since $g_p^{(2)}(X_p,X_p)$ is positive for all $X_p$, $A$ and $B$ are also positive.

   Any element of $TM$ can then be given by re-scaling an element of $SM$, and thus if we choose $C = \max \{\frac 1A , B\}$ we have the inequality
   $$\frac 1C g_p^{(1)}(X_p,X_p) \leq g_p^{(2)}(X_p,X_p)  \leq C g_p^{(1)}(X_p,X_p) $$
   for all $p \in M$ and all $X_p \in T_pM$. 

    The pullback of $g^{(1)}$ and $g^{(2)}$ to $\widetilde{M}$ defines a pair of $\Gamma$-invariant metrics given by 
   $$(\pi^* g^{(i)})_{\tilde p} (X_{\tilde p},Y_{\tilde p}) = g^{(i)}_{\pi(\tilde p)}(d_{\tilde p} \pi (X_{\tilde p}), d_{\tilde p}\pi(Y_{\tilde p})) $$
    for $i = 1, 2$ and for all $\tilde p\in \widetilde M$ and $X_{\tilde p},Y_{\tilde p}\in T_{\tilde p}\widetilde M$. Since $g^{(1)}$ and $g^{(2)}$ are quasi-isometric, it follow directly from the definition that $\pi^*g^{(1)}$ and $\pi^* g^{(2)}$ are also quasi-isometric.
\end{proof}

Now, we focus on the closed extensions of the Hilbert complex. Any metric on a compact manifold is complete. Taking the pullback onto the $\Gamma$-covering space $\widetilde{M}$ we see likewise that any $\Gamma$-invariant metric on $\widetilde{M}$ is complete.

From this we can conclude that the strong and weak extensions coincide for first order differential operators such as $ \delbar \oplus \del $ and $ \del+ \delbar$ in the ABC complex. See, \textit{e.g.}, \cite[Theorem 1.3]{A} for a proof. Unfortunately, this result cannot be applied to second order differential operators such as $\del\delbar$. Instead we will make use of the next Proposition, which can be applied to general elliptic operators on vector bundles, combined with Theorem \ref{thm bc essentially deldelbar}.


Let $P:\Gamma(M,E_1) \rightarrow \Gamma(M,E_2)$ be a differential operator between two vector bundles $E_1$ and $E_2$ on a smooth manifold $M$. 
Given a Galois $\Gamma$-covering, we can lift $P$ to a differential operator $\widetilde{P}$ between smooth sections of the induced pullback bundles $\widetilde{E_1}$ and $\widetilde{E_2}$ on the $\Gamma$-covering space $\widetilde{M}$ as follows. 
Any $\alpha\in\Gamma(\widetilde{M},\widetilde{E_1})$ can be locally written as $\pi^*\beta$ with $\beta\in\Gamma({M},{E_1})$, and this allows us to define the global operator $\widetilde{P}$ by its local action $\widetilde{P}\pi^*\beta:=\pi^*P\beta$. Basically the expression of $P$ in local coordinates gives a local formula for the lift $\widetilde{P}$.
Note that the differential operator $\widetilde{P}$ is \emph{$\Gamma$-equivariant}, in the sense that 
    $$ \gamma^*(\widetilde{P}\alpha)  = \widetilde{P}(  \gamma^*\alpha)  $$
    for all $\gamma \in \Gamma$ and $\alpha\in\Gamma(\widetilde{M},\widetilde{E_1})$. 
   
\begin{proposition}[\cite{Ati}, Proposition 3.1]\label{pullback self adj}
    Let $(E,h)$ be a Hermitian vector bundle on a smooth manifold $M$ and let $\pi:\widetilde{M} \rightarrow M$ be a Galois $\Gamma$-covering.
   Given an elliptic operator $P:\Gamma(M,E) \rightarrow \Gamma(M,E)$, denote its lift to $\widetilde{M}$ by $\widetilde{P}$. If $M$ is compact then $\widetilde{P}_s=\widetilde{P}_w$.
\end{proposition}
Given a Hermitian metric on $M$, let the Galois $\Gamma$-covering space $\widetilde{M}$ be endowed with the $\Gamma$-invariant pullback metric. Since $\pi^* d = d \pi^*$ and both the complex structure and the Hermitian metric on $\widetilde{M}$ are the pullback of the respective structures on $M$, it follows that $\pi^*\square_{BC}=\square_{BC}\pi^*$.
In particular the elliptic operator $\square_{BC}$ on $\widetilde{M}$ is just the lift of the same operator on $M$. By the above proposition we see that $(\square_{BC})_s=(\square_{BC})_w$ on $\widetilde M$, \textit{i.e.}, using Remark \ref{rmk strong-weak}, $(\square_{BC})_0$ is essentially self-adjoint on $\widetilde M$. Theorem \ref{thm bc essentially deldelbar} then tells us that the strong and weak extensions of $\del\delbar$ coincide, namely $\del\delbar_s=\del\delbar_w$.

From the above discussion we can now conclude the following.
\begin{theorem}
Let $\pi:\widetilde{M} \rightarrow M$ be a Galois $\Gamma$-covering with M a compact complex manifold. There exists a unique $L^2$ Aeppli-Bott-Chern complex on $\widetilde{M}$, which is determined by the induced complex structure and by any $\Gamma$-invariant metric on $\widetilde{M}$. 
   
The corresponding unreduced and reduced, Aeppli and Bott-Chern cohomologies, which we will denote respectively by
$$L^2 H_{A,\Gamma}^{\bullet,\bullet}(M), \quad L^2 H_{BC,\Gamma}^{\bullet,\bullet}(M) \quad \text{ and } \quad L^2\bar{H}_{A,\Gamma}^{\bullet,\bullet}(M), \quad L^2\bar{H}_{BC,\Gamma}^{\bullet,\bullet}(M) $$
are similarly unique and depend only on the Galois $\Gamma$-covering.
\end{theorem}



Given any Hermitian vector bundle $(E, h)$ on $M$, we denote the Hermitian pullback bundle on a Galois $\Gamma$-covering space $\widetilde{M}$ by $(\widetilde{E},\tilde{h})$, with $\widetilde{E}=\pi^*E$ and $\tilde h=\pi^*h$. Note that a Hermitian metric $\tilde h$ on $\widetilde E$ is the pullback of a Hermitian metric $h$ on $E$ iff $\tilde h$ is \emph{$\Gamma$-invariant}, \textit{i.e.}, if $\tilde h=\gamma^*\tilde h$ for all $\gamma\in\Gamma$, that is if
\[
\tilde h_{\tilde x}(\alpha(\tilde x),\alpha(\tilde x))=\tilde h_{\gamma(\tilde x)}\left(((\gamma^{-1})^*\alpha)(\gamma (\tilde x)),((\gamma^{-1})^*\alpha)(\gamma (\tilde x))\right)
\]
for all $\tilde x\in \widetilde M$, $\alpha$ a section of $\widetilde E$ and $\gamma\in \Gamma$.

The group $\Gamma$ then acts as isometries on the space of square-integrable sections $L^2 \widetilde{E}$ given by the pullback $\gamma^* \alpha$ for any $\gamma \in \Gamma$, $\alpha \in L^2 \widetilde{E}$.  

A closed Hilbert subspace $V \subseteq L^2 \widetilde{E}$ is called a \emph{Hilbert $\Gamma$-module} if it is preserved by the action of $\Gamma$, \textit{i.e.}, $\gamma^*V=V$ for any $\gamma \in \Gamma$. 
Maps between Hilbert $\Gamma$-modules are given by bounded \emph{$\Gamma$-equivariant} operators, namely bounded operators $P$ which commute with the pullback with respect to any $\gamma \in\Gamma$, \textit{i.e.}, $\gamma^*(P\alpha)=P(\gamma^*\alpha)$ for any $\alpha \in L^2 \widetilde{E}$. A map between Hilbert $\Gamma$-modules is called a \emph{weak isomorphism} if it is injective and has dense image. A \emph{strong isomorphism} is moreover an isometric isomorphism. It is well known, see, \textit{e.g.}, \cite[Lemma 2.5.3]{E}, that if there is a weak isomorphism between Hilbert $\Gamma$-modules then a strong isomorphism can be built by the polar decomposition. A sequence of maps between Hilbert $\Gamma$-modules
\[
\dots\longrightarrow V_{i-1}\overset{d_{i-1}}{\longrightarrow} V_i\overset{d_{i}}{\longrightarrow} V_{i+1}\longrightarrow\dots
\]
is called \emph{weakly exact} if $\c{\im d_{i-1}}=\ker d_i$ for all $i$.

We will now define the $\Gamma$-dimension of a Hilbert $\Gamma$-module $V \subseteq L^2 \widetilde{E}$. We refer to \cite{Ati,E,Lu} for a detailed treatment.

First take an orthonormal basis $\{\phi_i \}$ of $V$ and define the function on $\widetilde{M}$,
$$\tilde{f}(\tilde{x}) = \sum_i \tilde{h}_{\tilde x}(\phi_i (\tilde x),\phi_i (\tilde x)). $$
Note that the choice of orthonormal basis $\{\phi_i \}$ in the definition is not important. Indeed, given a different basis $\{\psi_j\}$ we can write
\begin{align*}
    \tilde{f}(\tilde{x}) &= \sum_i \tilde{h}_{\tilde x}(\phi_i (\tilde x),\phi_i (\tilde x))\\
    &= \sum_i \sum_j \tilde{h}_{\tilde x}(\phi_i(\tilde x), \psi_j(\tilde{x}) )\\
    &= \sum_j \tilde{h}_{\tilde x}(\psi_j (\tilde x),\psi_j (\tilde x)).
\end{align*}
Swapping the order of summation is allowed here as all of the summands are positive.
For any $\gamma\in\Gamma$ we have
\begin{align*}
\gamma^*\tilde{f}(\tilde{x})=\tilde{f}(\gamma(\tilde{x}))=\sum_i \tilde{h}_{\gamma(\tilde{x})}(\phi_i (\gamma(\tilde x)),\phi_i( \gamma(\tilde x))),
\end{align*}
and since 
$V$ is a Hilbert $\Gamma$-module then $\{(\gamma^{-1})^*\phi_i\}$ provides an orthonormal basis of $V$, so that
\begin{align*}
\gamma^*\tilde{f}(\tilde{x})=\sum_i \tilde{h}_{\gamma(\tilde{x})}\left(((\gamma^{-1})^*\phi_i)(\gamma(\tilde x)),((\gamma^{-1})^*\phi_i)(\gamma(\tilde x))\right)=\tilde{f}(\tilde{x}),
\end{align*}
using that $\tilde h$ is $\Gamma$-invariant.
Therefore $\tilde f$ descends to a well-defined function $f$ on $M$.
The $\Gamma$-dimension is then given by
$$\dim_{\Gamma} V = \int_M f(x) d\mu. $$

 Note that in the case when $\Gamma$ is the trivial group, the $\Gamma$-dimension simply counts the number of elements in the basis, \textit{i.e.}, it is the usual dimension of $V$ as a vector space.  Another immediate property of the $\Gamma$-dimension is that $\dim_\Gamma V=0$ if and only if $V=\{0\}.$

In principle, the value of $\tilde{f}$ may be infinite at some or all points in $\widetilde{M}$, in which case the $\Gamma$-dimension may also be infinite. However by the following proposition, we see that in certain cases we can guarantee that $\widetilde{f}$ is finite at all points, and thus when $M$ is compact, the $\Gamma$-dimension is finite.
\begin{proposition}[\text{\cite[Proposition 2.4]{Ati}}]\label{kernel dim finite}
Given a smooth manifold $\widetilde{M}$ along with a Hermitian vector bundle $(\widetilde{E},\tilde{h})$ and an elliptic differential operator acting on smooth sections $\widetilde{P}:\Gamma (\widetilde{M}, \widetilde{E}) \rightarrow \Gamma (\widetilde{M}, \widetilde{E})$, consider the kernel of the weak extension, $\ker \widetilde{P}_w \subseteq L^2 \widetilde{E}$. Given any orthonormal basis $\{\phi_i\}$ for $\ker \widetilde{P}_w$, the series
    $$\tilde{f}(\tilde x) = \sum_i \tilde{h}(\phi_i (\tilde x),\phi_i (\tilde x)),$$
    and all its derivatives, converge uniformly on compact subsets of $\widetilde{M}$. Consequently, $\tilde f$ is a smooth function. 
\end{proposition}
For any $\Gamma$-invariant Hermitian metric on $\widetilde{M}$, Theorem \ref{thm l2 bc decomp} implies the isomorphisms
$$L^2 \bar{H}_{A,\Gamma}^{p,q}(M) \cong  L^2 \mathcal{H}_{A,sw}^{p,q}, $$
$$L^2 \bar{H}_{BC,\Gamma}^{p,q}(M) \cong  L^2 \mathcal{H}_{BC,sw}^{p,q}.$$
Recall the definition of $L^2 \mathcal{H}_{A,sw}^{p,q}$ and $L^2 \mathcal{H}^{p,q}_{sw}$ as the kernels of $\square_{A,sw}$ and $\square_{BC,sw}$. Proposition \ref{pullback self adj} tells us that $\left(\square_A\right)_0$ and $\left(\square_{BC}\right)_0 $ are essentially self-adjoint operators. The choice of closed extensions $\square_{A,sw}$, $\square_{BC,sw}$ is therefore arbitrary but will be convenient later on. 

Since the differential operators $\del,\delbar,\del\delbar$ on $\widetilde M$ and their formal adjoints are $\Gamma$-equivariant, and the Hermitian metric on $\Lambda^{p,q}\widetilde M$ is $\Gamma$-invariant, it is easy to check that $L^2\H^{p,q}_{BC,sw}$ and $L^2\H^{p,q}_{A,sw}$ are Hilbert $\Gamma$-modules.

We can therefore define the \emph{$L^2$ Aeppli} and \emph{Bott-Chern numbers}, denoted by $h^{p,q}_{A,\Gamma}$ and $h^{p,q}_{BC,\Gamma}$, to be 
$$h^{p,q}_{A,\Gamma}(M) = \dim_{\Gamma} L^2 \mathcal{H}_{A,sw}^{p,q}, \quad \quad h^{p,q}_{BC,\Gamma}(M) = \dim_{\Gamma} L^2 \mathcal{H}_{BC,sw}^{p,q}. $$
By Proposition \ref{kernel dim finite} the Aeppli and Bott-Chern numbers are finite whenever $M$ is compact.

The space of Aeppli harmonic $L^2$ forms $L^2 \mathcal{H}^{p,q}_{A,sw}$ in general depends on the choice of $\Gamma$-invariant Hermitian metric. When necessary we will write $L^2 \mathcal{H}^{p,q}_{A,sw}(g)$ to denote the space corresponding to the metric $g$, and likewise for the space of Bott-Chern harmonic $L^2$ forms. Regardless of this, as a consequence of the next  proposition, the $L^2$ Aeppli and Bott-Chern numbers are still independent of the Hermitian metric on the compact complex manifold $M$.
We will need the following fundamental property of the $\Gamma$-dimension, which will be also used later on.

\begin{lemma}[\text{\cite[Corollary 3.4.6]{E}}]\label{lemma eckmann}
    Let 
    $$0\longrightarrow V_1 \longrightarrow V_2 \longrightarrow \dots \longrightarrow V_n \longrightarrow 0 $$
    be a weakly exact sequence of Hilbert $\Gamma$-modules. Then
    $$\sum_{k}(-1)^k \dim_{\Gamma} V_k = 0 $$
\end{lemma}

\begin{proposition}\label{prop gamma dim does not depend on the metric}
     Given any two $\Gamma$-invariant Hermitian metrics $g_1$ and $g_2$ on $\widetilde{M}$, we have
    $$ \dim_{\Gamma} L^2 \mathcal{H}^{p,q}_{A,sw}(g_1) = \dim_{\Gamma} L^2 \mathcal{H}^{p,q}_{A,sw}(g_2),  $$
    $$ \dim_{\Gamma} L^2 \mathcal{H}^{p,q}_{BC,sw} (g_1) = \dim_{\Gamma} L^2 \mathcal{H}^{p,q}_{BC,sw} (g_2).  $$
\end{proposition}
\begin{proof}
There exists an isomorphism $j:L^2 \mathcal{H}^{p,q}_{BC,sw}(g_1)\to L^2 \mathcal{H}^{p,q}_{BC,sw}(g_2)$, constructed via the isomorphism with the cohomology
$$L^2 \mathcal{H}^{p,q}_{BC,sw}(g_i) \simeq L^2 \bar H^{p,q}_{BC,\Gamma}$$ for $i = 1,2$.
In fact we now show that the isomorphism $j$ is actually realised by the projection
\[
\pi_{L^2 \mathcal{H}^{p,q}_{BC,sw}(g_2)}:L^2 \mathcal{H}^{p,q}_{BC,sw}(g_1)\to L^2\mathcal{H}^{p,q}_{BC,sw}(g_2).
\]
First we remark that $\c{\im\del\delbar_s}(g_1)=\c{\im\del\delbar_s}(g_2)$ since $g_1$ and $g_2$ are quasi isometric, therefore we can write $\c{\im\del\delbar_s}$ without any confusion. Any form $\alpha_1\in L^2 \mathcal{H}^{p,q}_{BC,sw}(g_1)$ can be decomposed, by Theorem \ref{thm l2 bc decomp}, as $\alpha_1=\alpha_2+\beta$, with $\alpha_2\in L^2\mathcal{H}^{p,q}_{BC,sw}(g_2)$ and $\beta\in\c{\im\del\delbar_s}$. Therefore $j$ first sends $\alpha_1$ in the class $[\alpha_1]_{L^2 \bar H^{p,q}_{BC,\Gamma}}$ which in turn is sent to $\alpha_2$. Thus $j(\alpha_1)=\alpha_2$. 

Our goal is to show that $j$ is a weak isomorphism of Hilbert $\Gamma$-modules, so that then by applying the previous Lemma we end the proof. Being $j$ an isomorphism, it is clearly injective and it has dense image. 
Moreover $j$ is a bounded operator since $g_1$ and $g_2$ are quasi isometric.

Finally, to see that $j$ is $\Gamma$-equivariant, we must show that $j(\gamma^*\alpha_1)=\gamma^*(j(\alpha_1))$ for all $\alpha_1\in L^2 \mathcal{H}^{p,q}_{BC,sw}(g_1)$. 
Since $\square_{BC}$ is $\Gamma$-equivariant, then  $\gamma^*\alpha_2\in L^2\mathcal{H}^{p,q}_{BC,sw}(g_2)$.
Therefore our claim is equivalent to showing that $\gamma^*\beta\in \c{\im\del\delbar_s}$. This is true since by Lemma \ref{lemma closure strong} the form $\beta$ is the $L^2$ limit of the smooth and compactly supported forms $\del\delbar\beta_j$ and the forms $\del\delbar\gamma^*\beta_j$ are as well smooth and compactly supported satisfying 
\[
\lv\gamma^*\beta-\del\delbar\gamma^*\beta_j\rv=\lv\gamma^*\beta-\gamma^*\del\delbar\beta_j\rv=\lv\beta-\del\delbar\beta_j\rv\to0.
\]
A similar weak isomorphism of Hilbert $\Gamma$-modules exists the same way between $L^2 \mathcal{H}^{p,q}_{A,sw}(g_i)$ for $i=1,2$.
\end{proof}

\begin{remark}
    By following the same reasoning as above but replacing the $L^2$ Aeppli-Bott-Chern Hilbert complex with the $L^2$ de Rham or Dolbeault Hilbert complex, it is possible to define the unique $L^2$ cohomology spaces $L^2 \bar{H}^k_{d, \Gamma}(M)$, $L^2 \bar{H}_{\del, \Gamma}^{p,q}(M)$ and $L^2 \bar{H}_{\delbar, \Gamma}^{p,q}(M)$, from which we obtain the $L^2$ Betti numbers $b^k_{\Gamma}(M)$, and the $L^2$ Hodge numbers, $h^{p,q}_{\del, \Gamma}(M)$ and $h^{p,q}_{\delbar, \Gamma}(M)$.
\end{remark}

Given a Galois $\Gamma$-covering space $\widetilde{M}$ of a complex manifold $M$, we define the following subspaces of $L^2 \Lambda^{\bullet, \bullet}$
\begin{alignat*}{2}
    & \mathcal{A}^{\bullet,\bullet} = \overline{\im \delbar_s} \cap \overline{\im \del_s} \cap \ker \delbar^* \del^*_w, & \quad 
    & \mathcal{B}^{\bullet,\bullet} = \ker \delbar_w \cap  \overline{\im \del_s} \cap \ker \delbar^* \del^*_w,\\
    & \mathcal{C}^{\bullet,\bullet} = \ker \del\delbar_w \cap \overline{\im \delbar^*_s} \cap \ker \del^*_w, & \quad 
    & \mathcal{D}^{\bullet,\bullet} = \overline{\im \delbar_s }\cap  \ker \del_w \cap \ker \delbar^* \del^*_w,\\
    & \mathcal{E}^{\bullet,\bullet} = \ker \del\delbar_w \cap \overline{\im \del^*_s} \cap \ker \delbar^*_w, & \quad 
    & \mathcal{F}^{\bullet,\bullet} = \ker \del\delbar_w \cap  \overline{\im \delbar^*_s} \cap \overline{\im \del^*_s}.
\end{alignat*}
These spaces actually are Hilbert $\Gamma$-modules. It is easy to check it using Lemma \ref{lemma closure strong} and the next facts: the differential operators $\del,\delbar,\del\delbar$ on $\widetilde M$ and their formal adjoints are $\Gamma$-equivariant; the Hermitian metric on $\Lambda^{p,q}\widetilde M$ is $\Gamma$-invariant.

\begin{remark}\label{isomorphisms}
 Note that there are isomorphisms induced by the identity between the spaces just defined and the following quotient spaces in $L^2 \Lambda^{\bullet, \bullet}$:
\begin{alignat*}{2}
  &\mathcal{A}^{\bullet,\bullet} \simeq A^{\bullet,\bullet} := \frac{\overline{ \im \delbar_s} \cap \overline{ \im \del_{s}}}{\overline{\im \del\delbar_s }},  & \quad&
  \mathcal{B}^{\bullet,\bullet} \simeq {B}^{\bullet,\bullet}:= \frac{\overline{ \im \del_s} \cap \ker \delbar_{w}}{\overline{\im \del\delbar_s} },\\ &
  \mathcal{C}^{\bullet,\bullet} \simeq C^{\bullet,\bullet}:= \frac{\ker \del\delbar_w}{\ker \delbar_w + \overline{\im \del_s}}, & \quad &
  \mathcal{D}^{\bullet,\bullet} \simeq D^{\bullet,\bullet}:= \frac{\overline{ \im \delbar_s} 
 \cap \ker \del_w }{\overline{\im \del\delbar_s }},\\ &  
 \mathcal{E}^{\bullet,\bullet} \simeq {E}^{\bullet,\bullet}:= \frac{{ \ker \del\delbar_w} }{{\ker} \del_w + \overline{\im \delbar_s} },   &\quad &
 \mathcal{F}^{\bullet,\bullet} \simeq {F}^{\bullet,\bullet} := \frac{\ker{\del\delbar_w}}{\ker \del_w + \ker \delbar_w }.
\end{alignat*}
\end{remark}

The following sequences are defined just by compositions of natural inclusions and projections (see the proof of the next proposition for the actual definitions of the maps in \eqref{exact sequence 1}; the maps in \eqref{exact sequence 2} are defined similarly)
\begin{equation}\label{exact sequence 1}
  0 \longrightarrow \mathcal{A}^{\bullet,\bullet} \longrightarrow  \mathcal{B}^{\bullet,\bullet} \longrightarrow L^2 \H_{\delbar,sw}^{\bullet,\bullet} \longrightarrow  L^2 \H_{A,sw}^{\bullet,\bullet}  \longrightarrow \mathcal{C}^{\bullet,\bullet} \longrightarrow 0 
\end{equation}
\begin{equation}\label{exact sequence 2}
      0 \longrightarrow \mathcal{D}^{\bullet,\bullet} \longrightarrow  L^2 \H_{BC,sw}^{\bullet,\bullet} \longrightarrow L^2 \H_{\delbar,sw}^{\bullet,\bullet} \longrightarrow  \mathcal{E}^{\bullet,\bullet}  \longrightarrow \mathcal{F}^{\bullet,\bullet} \longrightarrow 0 
\end{equation}
which actually are maps between Hilbert $\Gamma$-modules: it can be shown arguing similarly to the last part of the proof of Proposition \ref{prop gamma dim does not depend on the metric}.

We prove that the two sequences above are exact.
\begin{proposition}\label{prop sequences are exact}
The sequences \eqref{exact sequence 1} and $\eqref{exact sequence 2}$ are exact sequences of Hilbert $\Gamma$-modules. 
\end{proposition}
\begin{proof}
We just prove the exactness of the sequence \eqref{exact sequence 1} and omit the proof for \eqref{exact sequence 2} since it is very similar. We divide the proof in three steps.

Step 1): we build the maps of the sequence as a combination of projections and inclusions. 
\vspace{3mm}

$\mathcal{A}^{\bullet,\bullet} \longrightarrow \mathcal{B}^{\bullet,\bullet}$:
This map is simply an inclusion. 
\vspace{3mm}

$\mathcal{B}^{\bullet,\bullet} \longrightarrow L^2 \H_{\delbar,sw}^{\bullet,\bullet}$:
This map is given by composing the inclusion
\[
B^{\bullet,\bullet}\to \ker\delbar_w
\]
with the projection $$\ker \delbar_{w} \rightarrow L^2 \mathcal{H}_{\delbar,sw}^{\bullet,\bullet} = \ker \delbar_w \cap \ker \delbar^*_w.$$
\vspace{0mm}
    
$L^2\H_{\delbar,sw}^{\bullet,\bullet} \longrightarrow L^2\H_{A,sw}^{\bullet,\bullet}$:
This map is given by composing the inclusion
\[
L^2\H_{\delbar,sw}^{\bullet,\bullet}\to\ker\del\delbar_w
\]
with the projection $$\ker\del\delbar_w \rightarrow L^2\H_{A,sw}^{\bullet,\bullet}=\ker\del^*_w\cap\ker\delbar^*_w\cap\ker\del\delbar_w.$$
\vspace{0mm}

$L^2\H_{A,sw}^{\bullet,\bullet} \longrightarrow \mathcal{C}^{\bullet,\bullet}:$
This map is just a projection.
\vspace{3mm}

Step 2): there is a commutative diagram between sequences

\begin{equation*}
\begin{tikzcd}[row sep=small, column sep=small]
0\arrow[r]&\mathcal{A}^{\bullet,\bullet}\arrow[r]\arrow[d,"\simeq"]&\mathcal{B}^{\bullet,\bullet}\arrow[r]\arrow[d,"\simeq"]&L^2\mathcal{H}^{\bullet,\bullet}_{\delbar,sw}\arrow[r]\arrow[d,"\simeq"]&L^2\mathcal{H}^{\bullet,\bullet}_{A,sw}\arrow[r]\arrow[d,"\simeq"]&\mathcal{C}^{\bullet,\bullet}\arrow[r]\arrow[d,"\simeq"]&0\\
0\arrow[r]&{A}^{\bullet,\bullet}\arrow[r]&{B}^{\bullet,\bullet}\arrow[r]&L^2\bar{H}^{\bullet,\bullet}_{\delbar,sw}\arrow[r]&L^2\bar{H}^{\bullet,\bullet}_{A,sw}\arrow[r]&{C}^{\bullet,\bullet}\arrow[r]&0
\end{tikzcd}
\end{equation*}
where the maps of the sequence in the bottom line are just induced by the identity.

Step 3): the sequence in the bottom line is exact. It is straightforward to verify it once one uses the equivalent definition of $L^2\bar{H}^{\bullet,\bullet}_{A,sw}$ given by Corollary \ref{cor defin aeppli equivalent}.
\end{proof}

We refer to \cite[Section 3]{V} for the original idea behind the previous exact sequences. In \cite[Theorem A]{AT} Angella and Tomassini, starting from the exact sequences in \cite[Section 3]{V}, were able to prove an inequality on compact manifolds between the Aeppli, Bott-Chern and Dolbeault numbers,
$$h^{p,q}_{\del} + h^{p,q}_{\delbar} \leq h^{p,q}_{A} + h^{p,q}_{BC}.$$
We will conclude this section by proving the same inequality holds for the $L^2$ Aeppli, Bott-Chern and Dolbeault numbers on any Galois $\Gamma$-covering of a compact complex manifold. However we must first make a few comments about the exact sequences \eqref{exact sequence 1} and \eqref{exact sequence 2}.

We will write the $\Gamma$-dimension of the spaces $\mathcal{A}^{p,q}$, $\mathcal{B}^{p,q}$, $\mathcal{C}^{p,q}$, $\mathcal{D}^{p,q}$, $\mathcal{E}^{p,q}$ and $\mathcal{F}^{p,q}$ as $a^{p,q}$, $b^{p,q}$, $c^{p,q}$, $d^{p,q}$, $e^{p,q}$ and $f^{p,q}$.
We can see that $\mathcal{A}^{p,q}$, $\mathcal{B}^{p,q}$ and $\mathcal{D}^{p,q} $ are all contained within $L^2 \mathcal{H}^{p,q}_{BC,sw}$, while $\mathcal{C}^{p,q}$, $\mathcal{E}^{p,q}$ and $\mathcal{F}^{p,q}$ are contained within $L^2 \mathcal{H}^{p,q}_{A,sw}$.  Thus, by a simple property of the $\Gamma$-dimension given in the proposition below, we have $a^{p,q}, b^{p,q}, d^{p,q} \leq h_{A,\Gamma}^{p,q}(M)$ and $c^{p,q}, e^{p,q}, f^{p,q} \leq h_{BC,\Gamma}^{p,q}(M)$. In particular, if $M$ is compact, then $a^{p,q}$, $b^{p,q}$, $c^{p,q}$, $d^{p,q}$, $e^{p,q}$ and $f^{p,q}$ are all finite.
\begin{proposition}
    If $U,V \subset L^2 E$ are Hilbert $\Gamma$-modules such that $U \subset V$, then $\dim_{\Gamma} U \leq \dim_{\Gamma} V$.
\end{proposition}
\begin{proof}
    An orthonormal basis $\{\phi_i\}_{i \in I}$ of $V$ can be chosen such that $\{\phi_j\}_{j \in J}$  is an orthonormal basis of $U$, for some subset $J \subset I$. Clearly we have
    $$\sum_{i \in I} h_{\tilde x}(\phi_i(\tilde x),\phi_i(\tilde x)) \leq \sum_{j \in J} h_{\tilde x}(\phi_j(\tilde x),\phi_i(\tilde x))$$
    for all $\tilde x$ and so it follows from the definition that $\dim_{\Gamma} U \leq \dim_{\Gamma} V$.
\end{proof}

The following is the main result of this section.

\begin{theorem}\label{thm inequality bc numbers}
    For any Galois $\Gamma$-covering $\pi: \widetilde{M} \rightarrow M$ of a compact complex manifold $M$ the $L^2$ Aeppli, Bott-Chern and Dolbeault numbers satisfy the inequality
    $$ h^{p,q}_{\delbar, \Gamma}(M) + h^{p,q}_{\del, \Gamma}(M) \leq h^{p,q}_{A, \Gamma}(M) + h^{p,q}_{BC,\Gamma}(M). $$
\end{theorem}
\begin{proof}
    We will follow a similar argument to the proof of Theorem A in \cite{AT}.

    First note that by complex conjugation we obtain the equalities
    $$a^{p,q} = a^{q,p}, \quad b^{p,q} = d^{q,p}, \quad c^{p,q} = e^{q,p}, \quad f^{p,q} = f^{q,p}, $$
    $$ h^{p,q}_{A,\Gamma} = h^{q,p}_{A,\Gamma},  \quad h^{p,q}_{BC,\Gamma} = h^{q,p}_{BC,\Gamma}, \quad h^{p,q}_{\delbar, \Gamma} = h^{q,p}_{\del,\Gamma}.  $$

    From this, Proposition \ref{prop sequences are exact} and Lemma \ref{lemma eckmann}, we find that
    \begin{align*}
        h^{p,q}_{BC,\Gamma} + h^{p,q}_{A,\Gamma} &= h^{p,q}_{BC,\Gamma} + h^{q,p}_{A,\Gamma}\\
        &= h^{p,q}_{\delbar,\Gamma} + h^{q,p}_{\delbar,\Gamma} + a^{q,p} - b^{q,p}+ c^{q,p} + d^{p,q}  - e^{p,q} + f^{p,q} \\
        &= h^{p,q}_{\delbar,\Gamma} + h^{p,q}_{\del, \Gamma} + a^{p,q} + f^{p,q}\\
        &\geq  h^{p,q}_{\delbar,\Gamma} + h^{p,q}_{\del, \Gamma},
    \end{align*}
    which is the desired result.
\end{proof}

Considering that we have actually proved $h^{p,q}_{BC,\Gamma} + h^{p,q}_{A,\Gamma}= h^{p,q}_{\delbar,\Gamma} + h^{p,q}_{\del, \Gamma} + a^{p,q} + f^{p,q}$, and therefore in particular $ \sum_{p+q=k}h^{p,q}_{\delbar, \Gamma} + h^{p,q}_{\del, \Gamma} = \sum_{p+q=k}h^{p,q}_{A, \Gamma} + h^{p,q}_{BC,\Gamma}+a^{p,q} + f^{p,q}$, we immediately obtain the following two corollaries.

\begin{corollary}
For any Galois $\Gamma$-covering $\pi: \widetilde{M} \rightarrow M$ of a compact complex manifold $M$ it holds that
    $$h^{p,q}_{\delbar, \Gamma}(M) + h^{p,q}_{\del, \Gamma}(M) = h^{p,q}_{A, \Gamma}(M) + h^{p,q}_{BC,\Gamma}(M)$$
if and only if in $L^2\Lambda^{p,q}\widetilde{M}$
\[
\ker\del\delbar_w=\ker\del_w+\ker\delbar_w,\ \ \ \c{\im\del\delbar_s}=\c{\im\del_s}\cap\c{\im\delbar_s}.
\]
\end{corollary}
\begin{corollary}
For any Galois $\Gamma$-covering $\pi: \widetilde{M} \rightarrow M$ of a compact complex manifold $M$ it holds that
    $$\sum_{p+q=k}h^{p,q}_{\delbar, \Gamma}(M) + h^{p,q}_{\del, \Gamma}(M) = \sum_{p+q=k}h^{p,q}_{A,\Gamma}(M) + h^{p,q}_{BC,\Gamma}(M)$$
if and only if in $L^2\Lambda^{k}_\C\widetilde{M}$
\[
\ker\del\delbar_w=\ker\del_w+\ker\delbar_w,\ \ \ \c{\im\del\delbar_s}=\c{\im\del_s}\cap\c{\im\delbar_s}.
\]
\end{corollary}
Finally, if the compact complex manifold $M$ admits a K\"ahler metric, the corresponding pullback metric on $\widetilde M$ is complete and K\"ahler, therefore by Theorem \ref{thm kahler complete kernel} we deduce the following.
\begin{corollary}
For any Galois $\Gamma$-covering $\pi: \widetilde{M} \rightarrow M$ of a compact K\"ahler manifold $M$ for any $(p,q)$ it holds that
    $$ h^{p,q}_{\delbar, \Gamma}(M) + h^{p,q}_{\del, \Gamma}(M) = h^{p,q}_{A, \Gamma}(M) + h^{p,q}_{BC,\Gamma}(M). $$
\end{corollary}

\addtocontents{toc}{\protect\setcounter{tocdepth}{1}}

\section{Further remarks and open questions}\label{sec questions}

Let us present the full definition of the ABC complex and the $L^2$ ABC complex, which we have delayed until now since it was not really necessary for our purposes.

Given a complex manifold of dimension $n$ and a couple of integers $(p,q)$ such that $0\le p,q\le n$, we define the spaces
\begin{align*}
\mathcal{L}^{k}_{p,q}&:=\bigoplus_{r+s=k,r<p,s<q}A^{r,s}& &\text{ for }k\le {p+q-2},\\
\mathcal{L}^{k-1}_{p,q}&:=\bigoplus_{r+s=k,r\ge p,s\ge q}A^{r,s}& &\text{ for }k\le {p+q},
\end{align*}
as well as the differentials $\delta^k_{p,q}:\mathcal{L}^{k}_{p,q}\to\mathcal{L}^{k+1}_{p,q}$ with
\begin{align*}
\delta^k_{p,q}&=\pi_{\mathcal{L}^{k+1}_{p,q}}\circ d & & \text{ for }k\le p+q-3,\\
\delta^k_{p,q}&=\del\delbar & & \text{ for }k= p+q-2,\\
\delta^k_{p,q}&=d & &\text{ for }k\ge p+q-1,
\end{align*}
where $\pi_{\mathcal{L}^{k+1}_{p,q}}$ is the projection onto $\mathcal{L}^{k+1}_{p,q}$.
The full ABC complex is then
\[
0\longrightarrow\mathcal{L}^{0}_{p,q}\overset{\delta^0_{p,q}}{\longrightarrow}\mathcal{L}^{1}_{p,q}\overset{\delta^1_{p,q}}{\longrightarrow}\dots\overset{\delta^{2n-2}_{p,q}}{\longrightarrow}\mathcal{L}^{2n-1}_{p,q}\overset{\delta^{2n-1}_{p,q}}{\longrightarrow}\mathcal{L}^{2n}_{p,q}{\longrightarrow}0.
\]
We have already observed in Section \ref{sec complex manifold} that this complex is elliptic. In particular, given a Hermitian metric $g$ on $M$, the Laplacians $\Delta^k_{p,q}:\mathcal{L}^{k}_{p,q}\to\mathcal{L}^{k}_{p,q}$ defined by \eqref{eq defin ell lapl} are elliptic.
If $M$ is compact, it follows that
\[
\H^k_{p,q}:=\ker\Delta^k_{p,q}=\ker\delta^{k}_{p,q}\cap\ker(\delta^{k-1}_{p,q})^*,
\]
where $(\delta^{k-1}_{p,q})^*$ denotes the formal adjoint of $\delta^{k-1}_{p,q}$, and
all the dimensions
\[
h^k_{p,q}:=\dim_\C\H^k_{p,q}
\]
are finite (\textit{cf.} \cite{Ste}). 

We can similarly define the spaces of $L^2$ forms
\begin{align*}
L^2\mathcal{L}^{k}_{p,q}&:=\bigoplus_{r+s=k,r<p,s<q}L^2\Lambda^{r,s}& &\text{ for }k\le {p+q-2},\\
L^2\mathcal{L}^{k-1}_{p,q}&:=\bigoplus_{r+s=k,r\ge p,s\ge q}L^2\Lambda^{r,s}& &\text{ for }k\le {p+q},
\end{align*}
so that for $a_k\in\{s,w\}$ we have closed and densely defined extensions $(\delta^k_{p,q})_{a_k}$. The full $L^2$ ABC Hilbert complex is therefore
\[
0\longrightarrow L^2\mathcal{L}^{0}_{p,q}\overset{(\delta^0_{p,q})_{a_0}}{\longrightarrow} L^2\mathcal{L}^{1}_{p,q}
{\longrightarrow}\dots
{\longrightarrow} L^2\mathcal{L}^{2n-1}_{p,q}\overset{(\delta^{2n-1}_{p,q})_{a_{2n-1}}}{\longrightarrow} L^2\mathcal{L}^{2n}_{p,q}{\longrightarrow}0.
\]

\begin{remark}
    We chose to define the $L^2$ ABC Hilbert complex using just the strong and the weak extensions for more clarity. We point out that the same definition could be rephrased using the more general notion of ideal boundary condition introduced in \cite[Section 3]{BL}.
\end{remark}

We now list the following natural open questions arising in this context.

\subsection{Possibly incomplete K\"ahler metrics}\,\\
In Section \ref{sec complete kahler} we studied the $L^2$ Hodge theory for the Bott-Chern and Aeppli Laplacians on complex manifolds endowed with a complete K\"ahler metric.
It would be worthwhile to establish which results continue to be true after dropping the assumption of completeness.

\subsection{\texorpdfstring{$L^2$}{L2} Fr\"olicher inequality}\,\\
Given a compact complex manifold, in \cite[Theorems A,B]{AT} Angella and Tomassini have proved that
\[
h^{p,q}_{\del} + h^{p,q}_{\delbar} \leq h^{p,q}_{A} + h^{p,q}_{BC},
\]
and, taking into account the Fr\"olicher inequality 
\[
2b^k\le\sum_{p+q=k} h^{p,q}_{\del} + h^{p,q}_{\delbar}
\]
of \cite[Theorem 2]{Fro}, they also proved that
\[
2b^k=\sum_{p+q=k}h^{p,q}_{A} + h^{p,q}_{BC}
\]
if and only if the $\del\delbar$-Lemma holds. We are therefore motivated to ask the following question.

\emph{
Does a Fr\"olicher inequality between $L^2$ Betti and $L^2$ Hodge numbers
\[
b^k_\Gamma(M)\le\sum_{p+q=k}h^{p,q}_{\delbar,\Gamma}(M)
\]
hold?
}

In case of an affirmative answer it would follow that
\[
2b^k_\Gamma(M)\le\sum_{p+q=k}h^{p,q}_{A,\Gamma}(M) + h^{p,q}_{BC,\Gamma}(M)
\]
and one could study the limit case given by the equality, like in the case of compact complex manifolds.

\subsection{\texorpdfstring{$L^2$}{L2} index Theorem for the ABC complex}\,\\
The $L^2$ index theorem of Atiyah \cite{Ati} states that on a Galois $\Gamma$-covering of a compact manifold $M$, the analytical index of an elliptic operator $P$ between vector bundles coincides with the $\Gamma$-index of the lift operator $\tilde P$. When $P=d+d^*$, one finds that the Euler characteristic can be computed via $L^2$ Betti numbers
\[
\chi(M)=\sum_{k}(-1)^kb^k_{\Gamma}.
\]
If $M$ is a complex manifold and $P=\delbar+\delbar^*$, a similar result holds (\textit{cf.} \cite{B2}).

In our setting the natural operator to choose seems to be $P=\sum_{k}\delta^{2k}_{p,q}+\sum_{k}(\delta^{2k+1}_{p,q})^*$, however this is not elliptic. Therefore one would first need to generalise the $L^2$ index Theorem of Atiyah in the case of elliptic complexes.

We refer to \cite{Ste} for a recent result yielding the equality between the analytical and the topological indexes of the ABC complex $\mathcal{L}^\bullet_{p,q}$ on compact complex manifolds.

\end{document}